\documentclass{amsart}
\usepackage{mathabx}
\usepackage{graphicx}
\usepackage{amsfonts}
\usepackage{amsmath}
\usepackage{amscd}
\usepackage{amssymb}
\usepackage{verbatim}
\usepackage{hyperref}
\usepackage{marginnote}
\usepackage{tikz}
\usepackage{cite}
\usepackage{hyperref}
\usepackage{cleveref}
\usepackage{mathdots}

\newenvironment{proof*}{\noindent\emph{Proof}}{$\square$\smallskip}

\newtheorem{theorem}{Theorem}[section]
\newtheorem{problem}[theorem]{Problem}
\newtheorem{Definition}[theorem]{Definition}
\newtheorem{lemma}[theorem]{Lemma}
\newtheorem{Example}[theorem]{Example}

\newtheorem{corollary}[theorem]{Corollary}
\newtheorem{Remark}[theorem]{Remark}
\newtheorem{proposition}[theorem]{Proposition}

\newtheorem{Exercise}[theorem]{Exercise}
\newtheorem{Exercises}[theorem]{Exercises}
\newtheorem{Notation}[theorem]{Notation}
\newtheorem{Convention}[theorem]{Convention}
\newtheorem{standing assumption}[theorem]{Standing Assumption}

\newenvironment{definition}{\begin{Definition}\normalfont}{\end{Definition}}
\newenvironment{example}{\begin{Example}\normalfont}{\end{Example}}
\newenvironment{remark}{\begin{Remark}\normalfont}{\end{Remark}}

\newenvironment{convention}{\begin{Convention}\normalfont}{\end{Convention}}

\title[Diagram groups and groups of PL homeomorphisms  with fixed points]{Diagram groups and groups of piecewise linear homeomorphisms of the line with global fixed points}  
\author[D.~S.~Farley]{Daniel S. Farley}
\address{Department of Mathematics\\ Miami University\\ Oxford, OH 45056 U.S.A.}
\email{farleyds@miamioh.edu}

\date{\today}

\begin{document}
\begin{abstract}     
Assume $n \geq 2$ and $\ell = (r_{1}, \ldots, r_{k}) \in [0,1]^{k}$ is an increasing sequence of real numbers. Let $G_{n,\ell}$ denote the 
group of orientation-preserving piecewise linear homeomorphisms $h$
of $I = [r_{1}, r_{k}]$ such that: (i) $h'(x)$ is a power of $n$ where it is defined; (ii) if $h'(x)$ is undefined, then $x$ is an $n$-adic rational number, (iii) $h$ fixes each entry of $\ell$, and (iv) $h(\mathbb{Z}[1/n] \cap I) = 
\mathbb{Z}[1/n] \cap I$. 

We prove that $G_{n,\ell}$ is a diagram group $D(\mathcal{P}_{n,\ell}, \omega_{n,\ell})$ for all integers 
$n \geq 2$ and for all finite sequences $\ell$. The semigroup presentation
$\mathcal{P}_{n,\ell}$ and the word $\omega_{n,\ell}$ can be computed from the $n$-ary expansions of the numbers $r_{i}$. If all entries in $\ell$ are rational, then $G_{n,\ell}$ has type $F_{\infty}$. Otherwise, $G_{n,\ell}$ is not finitely generated. 
\end{abstract}



\maketitle


\section{Introduction}
In \cite{BS}, Bieri and Strebel studied groups of piecewise linear homeomorphisms of the line, denoted by $G(I;A,P)$. These groups are defined by three parameters: (i) a closed interval $I \subseteq \mathbb{R}$; (ii) a multiplication subgroup 
$P$ of the positive real numbers, and (iii) a $\mathbb{Z}[P]$-submodule $A$ of $\mathbb{R}$. The group $G(I;A,P)$ is the collection of all piecewise linear homeomorphisms $h$ of $I$ such that:
\begin{enumerate}
\item $h(I \cap A) =I \cap A$;
\item the derivative $h'(x)$, where defined, is always in $P$, and
\item the points $x$ where $h'(x)$ is undefined lie in $A$.
\end{enumerate} 
For instance, these definitions imply that
the generalized Thompson group $F_{n}$ ($n \geq 2$) is $G([0,1]; \mathbb{Z}[1/n], \langle n \rangle)$. 

One somewhat unusual feature of \cite{BS} is that the authors consider the role of the interval $I$ in great depth. Following their example, this author began to consider the groups $G([a,b]; \mathbb{Z}[1/n], \langle n \rangle)$ for arbitrary $a$ and $b$, with the idea of studying their topological finiteness properties $F_{n}$. (Here we recall that \cite{BG} and \cite{B} proved, respectively, that the Thompson groups $F_{2}$ and $F_{n}$ have type $F_{\infty}$.)
After beginning this project, the author learned of \cite{GolanSapir} and \cite{Savchuk}. In \cite{GolanSapir}, Golan and Sapir showed that the stabilizer in $F_{2}$ of a sequence
$0 < r_{1} < r_{2} < r_{3} < \ldots < r_{m} < 1$ of rational numbers is always finitely generated and that 
the isomorphism types of these stabilizers depend only on $m$ and the 
set $\{ i \mid r_{i} \text{ is a dyadic rational } \}$. In \cite{Savchuk}, it was shown 
that the stabilizer in $F_{2}$ of a single rational number is always finitely presented, and the authors explicitly computed a finite presentation.

In view of the results from \cite{BS}, \cite{GolanSapir} and \cite{Savchuk}, it is natural to consider the following class of groups.
\begin{definition} \label{definition:ourgroups}
(The groups $G_{n,\ell}$) 
Let $n \geq 2$ and let $\ell = (r_{1}, \ldots, r_{k}) \in [0,1]^{k}$
be an increasing sequence. We let $G_{n,\ell}$ denote the subgroup 
of $G([r_{1},r_{k}]; \mathbb{Z}[1/n], \langle n \rangle)$ that fixes $\ell$.
\end{definition}
We will prove the following theorem.
\begin{theorem} \label{theorem:bigone}
The group $G_{n,\ell}$ is isomorphic to a diagram group $\mathrm{D}(\mathcal{P}_{n,\ell},\omega_{n,\ell})$, where the semigroup presentation $\mathcal{P}_{n,\ell}$ and base word $\omega_{n,\ell}$ can be computed algorithmically from the $n$-ary expansions of the members of $\ell$.

If $\ell \in \mathbb{Q}^{k}$, then $\mathcal{P}_{n,\ell}$ is a finite semigroup presentation 
and $G_{n,\ell}$ is of type $F_{\infty}$. If any member of $\ell$ is irrational, then $\mathcal{P}_{n,\ell}$ is infinite and $G_{n,\ell}$ is
not finitely generated.
\end{theorem}

The $F_{\infty}$ statement extends some of the above-mentioned results from \cite{GolanSapir} and \cite{Savchuk} to arbitrary $n$, since $F_{\infty}$ groups are finitely presentable. The restriction of entries of $\ell$ to $[0,1]$ is done without loss of generality, since, if $\ell \in \mathbb{R}^{k}$, then there is an affine transformation of $\mathbb{R}$ conjugating $G_{n,\ell}$ to $G_{n,\ell'}$
for some $\ell' \in [0,1]^{k}$.

The theory of diagram groups was developed initially by Kilibarda in her thesis \cite{Kilibarda}, and later by Guba and Sapir \cite{GS}. A major impetus for the development of this theory was Guba's observation \cite{GS} that Thompson's group $F_{2}$ is a diagram group. Our main theorem is a generalization of Guba's observation, and raises the possibility of applying diagram group methods to the study of the groups $G_{n,\ell}$. We do not attempt to summarise such methods here, but refer the interested reader to \cite{GS} and the recent survey \cite{Genevois}, which features an extensive bibliography. We can mention the following corollary.
\begin{corollary} \label{corollary:mycorollary}
For arbitrary $n$ and $\ell$:
\begin{enumerate}
\item The integral homology groups $H_{k}(G_{n,\ell})$ are free abelian, and
\item The group $G_{n,\ell}$ acts properly by isometries on a locally finite
CAT(0) cubical complex. 
\end{enumerate}
\end{corollary}
Part (1) is due to results of \cite{Genevois, GS2} and (2) is due to results of \cite{Fthesis}.

The proof of Theorem \ref{theorem:bigone} uses methods that were developed by Hughes and the author in \cite{FH}, and later refined by the author in
\cite{Far}. For a given $G_{n,\ell}$, the results of \cite{FH} and \cite{Far} produce a simplicial complex $\Delta_{n,\ell}$ on which 
$G_{n,\ell}$ acts freely and cellularly. 
The main observation in the proof of Theorem \ref{theorem:bigone} is 
that the quotient of the $2$-skeleton $\Delta_{n,\ell}^{(2)}$ by the action of $G_{n,\ell}$ is isomorphic to a certain singular square complex $S(\mathcal{P}_{n,\ell}, \omega_{n,\ell})$ defined in \cite{GS}, called the \emph{Squier complex}. The group
$\pi_{1}(S(\mathcal{P}_{n,\ell}, \omega_{n,\ell}))$ is isomorphic to
$D(\mathcal{P}_{n,\ell}, \omega_{n,\ell})$ by results of \cite{GS}. Thus, we have the following sequence of isomorphisms: 
\[ G_{n,\ell} \, \cong \, \pi_{1}(G_{n,\ell}\backslash \Delta_{n,\ell}) \, \cong \, 
\pi_{1}(S(\mathcal{P}_{n,\ell}, \omega_{n,\ell})) \, \cong \, D(\mathcal{P}_{n,\ell},\omega_{n,\ell}), \]
from which the isomorphism result of Theorem \ref{theorem:bigone} follows. The finiteness results of Theorem \ref{theorem:bigone} will be deduced from a general framework that was described by the author in \cite{Far}. 

Our proof will be as self-contained as possible, but we will nevertheless need to use results from \cite{Far} and \cite{Far2} rather extensively. Very little background in diagram groups is assumed or needed, since we will effectively define a diagram group as the fundamental group of the relevant Squier complex.

After an earlier version of this paper was uploaded to the arXiv, the author learned of some overlap with previous results. The reader is referred to Section \ref{section:alternate} for complete details, but we should note here that Golan-Polak and Sapir \cite{GolanSapir2} previously observed that the groups $G_{n,\ell}$ are diagram groups, anticipating our main result. 

We now briefly outline the paper.
In Section \ref{section:diagram}, we introduce semigroup presentations $\mathcal{P}$ and their associated diagram groups
$D(\mathcal{P},\omega)$, where $w$ is a base word. In Section \ref{section:example}, we walk the reader through a procedure which, given $n$ and $\ell$, produces a semigroup presentation $\mathcal{P}_{n,\ell}$ and word $\omega_{n,\ell}$ such that $G_{n,\ell}$
is isomorphic to $D(\mathcal{P}_{n,\ell}, \omega_{n,\ell})$ . The justification of the procedure from Section \ref{section:example} begins in Section \ref{section:local} and ends in Section \ref{section:isomorphism}. Section \ref{section:local} introduces the concept of ``local generators" for a generalized Thompson group, and considers in detail the base case in which $\ell = (a,b)$. Section \ref{section:genlocal} describes local generators for the groups $G_{n,\ell}$ using the device of a labelled tree $T_{n,\ell}$.
Section \ref{section:CAT(0)} describes the construction of a CAT(0) cubical complex $\Delta(n,\ell)$ on which the group $G_{n,\ell}$ acts, and proves that the action is always free. Section \ref{section:isomorphism} concludes the proof of the isomorphism statement in Theorem \ref{theorem:bigone} by showing that the  $2$-skeleton of $G_{n,\ell} \backslash \Delta(n,\ell)$  is isomorphic to a connected component of the Squier complex for $\mathcal{P}_{n,\ell}$. Section \ref{section:finf} proves the $F_{\infty}$ property for the case in which $\ell$ is a list of rational numbers, and shows that $G_{n,\ell}$ is not finitely generated if $\ell$ has any irrational entries. The proof uses Brown's Finiteness Criterion \cite{B} and a framework from \cite{Far}. Section \ref{section:alternate} sketches some of the overlap with work of other authors and some additional directions for research.

\section{Diagram groups} \label{section:diagram}

\begin{definition} (Semigroup presentations; equivalence modulo a presentation)
Let $\Sigma$ be an alphabet. A \emph{positive word} in the alphabet $\Sigma$ is a string of symbols from $\Sigma$ without any occurrences of inverses (e.g., if $\Sigma = \{ a, b \}$, then $aba$ is positive, but $ab^{-1}a$ is not). The \emph{free monoid} $\Sigma^{\ast}$ on $\Sigma$ is the collection of all positive words in the symbols $\Sigma$, including the empty word, which we denote by $\varepsilon$. The \emph{free semigroup} $\Sigma^{+}$ on 
$\Sigma$ consists of all non-empty positive words in the symbols $\Sigma$. In both $\Sigma^{\ast}$ and $\Sigma^{+}$, the operation is concatenation of words.

We write $\omega' \equiv \omega''$ if $\omega'$ and $\omega''$ are identical, letter for
letter, as words. 

A \emph{semigroup presentation} $\mathcal{P}$ is a pair $\langle \Sigma \mid \mathcal{R} \rangle$, where $\Sigma$ is an alphabet and $\mathcal{R}$ is a collection of ordered pairs $(r_{1}, r_{2}) \in \Sigma^{+} \times \Sigma^{+}$. Two words $\omega', \omega'' \in \Sigma^{+}$ are \emph{equivalent
modulo $\mathcal{P}$} if there is a sequence
\[ \omega' \equiv \omega_{1}, \ldots, \omega_{m} \equiv \omega'' \]
such that, for each $i \in \{1, \ldots, m-1 \}$, there is an ordered pair $(r_{1},r_{2}) \in \mathcal{R}$ such that $\omega_{i+1}$ is the result of replacing some occurrence of either $r_{1}$ or $r_{2}$ in $\omega_{i}$ with (respectively) $r_{2}$ or $r_{1}$. We write $\omega' =_{\mathcal{P}} \omega''$, or simply $\omega' = \omega''$ if the presentation $\mathcal{P}$ is clear from the context.  Equivance module $\mathcal{P}$ is an equivalence relation on $\Sigma^{+}$. We denote the equivalence class of a word $\omega$ by  $[\omega]$.

We will write $r_{1} \rightarrow r_{2}$ for members of $\mathcal{R}$, in place of $(r_{1},r_{2})$.
\end{definition}

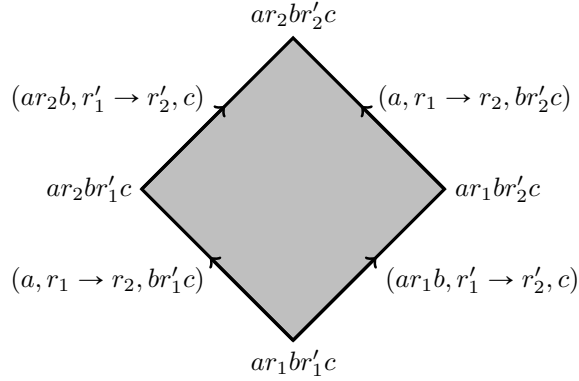
\begin{figure}[!t] 
\begin{center} 
\begin{tikzpicture}

\filldraw[lightgray,thick] (0,0) -- (2,2) -- (0,4) -- (-2,2) --  cycle; 
\draw[black, very thick] (0,0) -- (2,2) -- (0,4) -- (-2,2) -- cycle; 
\draw[black, very thick, ->] (0,0) -- (1.1,1.1); 
\node at (2.5,.8){$(ar_{1}b,r'_{1} \rightarrow r'_{2},c)$}; 
\draw[black, very thick, ->] (-2,2) -- (-.9,3.1); 
\node at (-2.45,3.2) {$(ar_{2}b,r'_{1}\rightarrow r'_{2},c)$};
\draw[black, very thick, ->] (0,0) -- (-1.1,1.1); 
\node at (-2.45,.8){$(a,r_{1} \rightarrow r_{2},br'_{1}c)$};
\draw[black, very thick, ->] (2,2)-- (.9,3.1); 
\node at (2.4,3.2){$(a,r_{1} \rightarrow r_{2},br'_{2}c)$};

\node at (0,-.3){$ar_{1}br'_{1}c$}; 
\node at (2.7,2){$ar_{1}br'_{2}c$}; 
\node at (-2.7,2){$ar_{2}br'_{1}c$}; 
\node at (0,4.3){$ar_{2}br'_{2}c$}; 
\end{tikzpicture}
\end{center}
\caption{The square $(a,r_{1} \rightarrow r_{2}, b, r'_{1} \rightarrow r'_{2}, c)$ in a Squier complex.}
\label{squier}
\end{figure}

\begin{definition}  \cite{GS} (The Squier complex of a semigroup presentation)
Let $\mathcal{P} = \langle \Sigma \mid \mathcal{R} \rangle$ be a semigroup presentation.
The \emph{Squier complex} $S(\mathcal{P})$ of $\mathcal{P}$ is a singular square complex defined as follows.
The set of vertices of $S(\mathcal{P})$ is $\Sigma^{+}$.  A directed edge is denoted by a triple
$(a, r_{1} \rightarrow r_{2}, b)$, where $(r_{1} \rightarrow r_{2}) \in \mathcal{R}$ and
$a,b \in \Sigma^{\ast}$. The directed edge in question has $ar_{1}b$ as its initial vertex and  $ar_{2}b$ as its terminal vertex. Squares take the general form
$(a,r_{1}\rightarrow r_{2}, b, r'_{1} \rightarrow r'_{2}, c)$, where 
$(r_{1} \rightarrow r_{2})$, $(r'_{1} \rightarrow r'_{2}) \in \mathcal{R}$ and $a,b,c \in \Sigma^{\ast}$. The attaching map of a typical square is indicated in Figure \ref{squier}. 

If $\omega \in \Sigma^{+}$, then we let $S(\mathcal{P},\omega)$ denote the connected component of $S(\mathcal{P})$ containing the vertex $\omega$.
\end{definition}

\begin{definition} \cite{GS} (The diagram group $D(\mathcal{P},\omega)$) \label{definition:Dpw}
Let $\mathcal{P} = \langle \Sigma \mid \mathcal{R} \rangle$ be a semigroup presentation and let $\omega \in \Sigma^{+}$. 
The diagram group $D(\mathcal{P},\omega)$ is the fundamental group of
$S(\mathcal{P},\omega)$ with basepoint at $\omega$.
\end{definition}

\begin{remark}
Assume that a semigroup presentation $\mathcal{P}$ is given.
Guba and Sapir \cite{GS} describe how to make a certain class of oriented labelled planar graphs, called \emph{semigroup diagrams over $\mathcal{P}$}, into a group with respect to a natural stacking operation.
Strictly speaking, one needs to fix a base word $\omega \in \Sigma^{+}$, and then the semigroup diagrams over $\mathcal{P}$ whose top and bottom labels are $\omega$ become a group, which is denoted by $D(\mathcal{P},\omega)$. It is a theorem of \cite{GS} that $D(\mathcal{P},\omega)$ is isomorphic to $\pi_{1}(S(\mathcal{P},\omega))$. 

In the main argument of this paper, we can avoid any mention of semigroup diagrams and simply use the isomorphism
$D(\mathcal{P},\omega) \cong \pi_{1}(S(\mathcal{P},\omega))$ as a definition
of $D(\mathcal{P},\omega)$ (as above). 
\end{remark}

\section{An example} \label{section:example} 

In this section, we consider the group $G_{3,\ell}$, where $\ell = (5/26, 7/26, 5/6)$. We will completely describe how to produce an associated semigroup presentation $\mathcal{P}$ and a base word $\omega$ such that $G_{3, \ell} \cong \mathcal{D}(\mathcal{P},\omega)$. 

\subsection{$3$-ary expansions and the associated automata} \label{subsection:auto}

The first step in describing $G_{3,\ell}$ as a diagram group is to find $3$-ary expansions for the entries of $\ell$. A straightforward calculation shows that
\begin{align*}
\displaystyle 5/26 &= .\overline{012},  \\
\displaystyle 7/26 &= .\overline{021},  \text{ and }\\
\displaystyle 5/6 &= .2\overline{1}.
\end{align*}

\begin{remark} \label{remark:nary}
More generally, if $n \geq 2$ and $\ell = (r_{1}, \ldots, r_{k})$, we find $n$-ary expansions of each entry of $\ell$ subject to certain constraints:
\begin{enumerate}
\item if $r_{i}$ ($i < k$) admits a terminating $n$-ary expansion, then we choose the (unique) expansion that terminates with an infinite sequence of $0$s;
\item if $r_{k}$ admits a terminating $n$-ary expansion, then we choose the (unique) expansion that terminates with an infinite sequence of $n-1$s;
\item we choose a representation in which the overlined subsequence is not a proper power;
\item we choose the expansion of each $r_{i}$ in such a way that the final symbol under the overline is never equal to the final symbol preceding the overlined subsequence.
\end{enumerate}
The latter condition can always be satisfied. For example, instead of $.01\overline{121}$, we could substitute the expansion $.0\overline{112}$, and repeat the indicated procedure as necessary.
\end{remark}

With each $3$-ary expansion, we now associate a labelled directed graph $A_{n}(r_{i})$ called an  \emph{$n$-ary automaton} (or \emph{automaton} if the $n$ is clear from the context). Each automaton has a basepoint $v$. A directed path $p$, labelled by the non-repeating symbols in the expansion, connects $v$ to a second vertex $v'$. The length of $p$ is equal to the total number of digits in the non-repeating portion of the expansion, and each directed edge is labelled by a single digit. At $v'$, we attach a directed loop, similarly labelled by the overlined symbols in the expansion.
The $n$-ary automata for $5/26$, $7/26$, and $5/6$ are depicted in Figure \ref{auto}. (The alphabetical labels will be explained in the next subsection.)

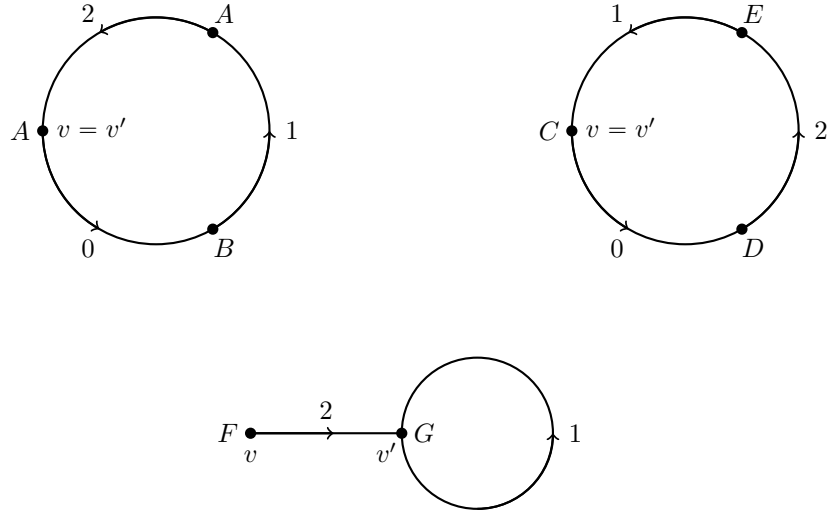
\begin{figure}[!b] 
\begin{center}
\begin{tikzpicture}

\draw[thick] (0,0) circle (1.5); 

\filldraw[thick] (-1.5,0) circle (.06);   
\draw[thick,->] (-1.5,0) arc (180:240:1.5);  
\node at (-1.8,0){$A$};
\node at (-.85,0.05){$v=v'$};

\filldraw[thick] (.75,-1.3) circle (.06); 
\draw[thick, ->] (.75,-1.3) arc (-60:0:1.5);
\node at (.9, -1.55){$B$};

\filldraw[thick](.75,1.3) circle (.06); 
\draw[thick, ->](.75,1.3) arc (60:120:1.5);
\node at (.9, 1.55){$A$};

\node at (-.9,-1.55){$0$}; 
\node at (1.8,0){$1$}; 
\node at (-.9,1.55){$2$};

\draw[thick] (7,0) circle (1.5); 

\filldraw[thick] (5.5,0) circle (.06);   
\draw[thick,->] (5.5,0) arc (180:240:1.5);  
\node at (5.2,0){$C$};
\node at (6.15,0.05){$v=v'$};

\filldraw[thick] (7.75,-1.3) circle (.06); 
\draw[thick, ->] (7.75,-1.3) arc (-60:0:1.5);
\node at (7.9, -1.55){$D$};

\filldraw[thick](7.75,1.3) circle (.06); 
\draw[thick, ->](7.75,1.3) arc (60:120:1.5);
\node at (7.9, 1.55){$E$};

\node at (6.1,-1.55){$0$}; 
\node at (8.8,0){$2$}; 
\node at (6.1,1.55){$1$};


\draw[thick](4.25,-4) circle (1); 

\filldraw[thick](1.25,-4) circle (.06); 
\node at (.95,-4){$F$}; 
\node at (1.25,-4.3){$v$}; 
\filldraw[thick](3.25,-4) circle (.06); 
\node at (3.55,-4){$G$}; 
\node at (3.05,-4.25){$v'$}; 

\draw[thick](1.25,-4) -- (3.25,-4); 
\node at (2.25,-3.7){$2$};

\draw[thick, ->](1.25,-4) -- (2.35,-4); 
\draw[thick, ->](4.25,-5) arc (-90:0:1); 
\node at (5.55,-4){$1$};

\end{tikzpicture}
\end{center}
\caption{The $n$-ary automata for the numbers $5/26$, $7/26$, $5/6$ are depicted here, in clockwise order from the top left.}
\label{auto}
\end{figure}

\begin{remark}
In the case that $r_{i}$ has a non-repeating expansion (because $r_{i}$ is irrational), the automaton simply consists of an infinite directed ray, labelled by the digits in the expansion of $r_{i}$.

If the expansion of $r_{i}$ is purely periodic, then $v = v'$ and $p$ is empty.
\end{remark}

\begin{remark}
Note that, in view of conditions (1) and (2) above, the $n$-ary automaton of a number $r_{i}$ with a terminating $n$-ary expansion is not unique, due to the possible dependence on $i$. In what follows we will always, however, work with a specific sequence $\ell$ in mind, and within this context the $k$-tuple
$(A_{n}(r_{1}), \ldots, A_{n}(r_{k}))$ is always uniquely determined by $\ell$. We can therefore speak of ``the" automaton $A_{n}(r_{i})$ without fear of ambiguity. 
\end{remark}

\begin{remark} \label{remark:minimal} (minimality of the chosen representative)
If $.\alpha \overline{\beta}$ is an $n$-ary expansion of $r_{i}$ satisfying constraints (1) and (2) from Remark \ref{remark:nary}, and $.\gamma \overline{\delta}$ 
satisfies (1)-(4), then $|\alpha| \geq |\gamma|$ and $\beta$ is a cyclic shift of some power of $\delta$. In particular, $|\beta| = m|\delta|$, for some $m>0$.  
\end{remark}

\subsection{Labellings of the automata and semigroup presentations}
\label{subsection:labels} 

Next we label the vertices of the automata from Subsection \ref{subsection:auto}.  
We first define equivalence relations on the vertices of each individual automaton.
In the automaton for $5/26$, we define two vertices $v_{1}$ and $v_{2}$ to be equivalent 
if there is a directed edge-path from $v_{1}$ to $v_{2}$ (or from $v_{2}$ to $v_{1}$) that is labelled by a sequence of $2$s. In the automaton for $5/6$, we define two vertices $v_{1}$ and $v_{2}$ to be equivalent if there is a directed edge-path from $v_{1}$ to $v_{2}$ (or from $v_{2}$ to $v_{1}$) that is labelled by a sequence of $0$s. In the automaton for $7/26$, two vertices are equivalent only if they are identical.

\begin{remark}
The general situation is as follows. Let $\ell = (r_{1}, \ldots, r_{k})$. Two vertices are equivalent in 
the automaton for $r_{1}$ if they are connected by a directed path labelled by $n-1$s. Two vertices are equivalent in the automaton for $r_{k}$ if they are connected by a directed path labelled by $0$s. In the remaining automata (for the $r_{i}$ with $1 < i < k$) no two distinct vertices are equivalent.
\end{remark}

Now we introduce labels for all of the vertices. Two vertices get different labels if and only if they are not equivalent. The alphabetical labellings of the vertices in Figure \ref{auto} follow this rule. 

In $A_{3}(5/26)$ (more generally, in $A_{n}(r_{1})$), each equivalence class contains a unique vertex whose outgoing edge is not labelled by $2$ (respectively, by $n-1$). We call such a vertex \emph{reduced}.
In $A_{3}(5/6)$ (more generally, in  $A_{n}(r_{k})$), each equivalence class contains a unique vertex whose outgoing edge is not labelled by $0$. Such vertices are also said to be \emph{reduced}. In $A_{3}(7/26)$ (more generally, in the automata $A_{n}(r_{i})$, $1<i<k$), all vertices are considered to be reduced.

We can now read a semigroup presentation $\mathcal{P}_{3,\ell}$ from the labelled automaton. We assume that `$x$' is not among the vertex labels selected above. Let $\Sigma_{3,\ell}$ be the collection of all vertex labels and 
the additional symbol $x$. We define $\mathcal{R}_{3,\ell}$ as follows:
\begin{enumerate}
\item For each reduced vertex in the automata for the $r_{i}$ ($1<i<k$), we introduce 
a relation $(X, x^{i}Yx^{n-i-1})$, where $X$ is the label of the reduced vertex in question, $i$ is the number labelling the outgoing arrow, and $Y$ is the label of the target vertex;
\item for each reduced vertex in the automaton for $r_{1}$, we introduce a relation
$(X,Yx^{n-i-1})$, where $X$, $Y$, and $i$ are as above;
\item for each reduced vertex in the automaton for $r_{k}$, we introduce a relation
$(X,x^{i}Y)$, where $X$, $Y$, and $i$ are as above;
\item we add the additional relation $(x,x^{n})$.
\end{enumerate} 

Following the above procedure for the automata in Figure \ref{auto}, we conclude that
$\mathcal{P}_{3,\ell} = \langle \Sigma_{n,\ell} \mid \mathcal{R}_{n,\ell} \rangle$, where
\begin{align*}
\Sigma_{n,\ell} &= \{ A, B, C, D, E, F, G, x \} \\
\mathcal{R}_{n,\ell} &= \{ (A, Bxx), (B,Ax), (C,Dxx), (D,xxE), \\
& \quad \quad (E,xCx), (F,xxG), (G,xG), (x,xxx) \}
\end{align*}

\subsection{Determination of the basepoint $\omega_{n,\ell}$} \label{subsection:basepoint}

Let $T_{n}$ denote the rooted ordered infinite $n$-ary tree. Each node in the tree has an associated \emph{address}
from $\{0, \ldots, n-1 \}^{\ast}$, 
which we can describe as follows. The address of the root is the empty string $\varepsilon$. More generally, if
the address $\omega$ of a vertex $v$ is given, then the addresses of the children of $v$ are, from left to right, 
$\omega 0$, $\ldots$, $\omega (n-1)$ .  

We say that an automaton \emph{accepts} an address if there is a directed edge-path from the basepoint 
that is labelled by $\omega$.

We choose a finite rooted $n$-ary tree such that the address of each leaf is accepted by at most one of the automata $A_{n}(r_{i})$. If the address of a leaf $z$ is accepted by an automaton $A_{n}(r_{i})$, then there is a directed path $p$ in $A_{n}(r_{i})$ issuing from the basepoint such that the label of $p$ matches the address of $z$. Let $X$ denote the label in $A_{n}(r_{i})$ of the terminal vertex $\tau(p)$. We label $z$ by $X$.
Now, beginning with the leftmost labelled leaf and ending with the rightmost labelled leaf, we give each previously unlabelled leaf the label $x$.

The result of applying this procedure to our running example is depicted in Figure \ref{mytree}.

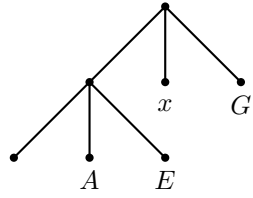
\begin{figure}[!h] 
\begin{center}
\begin{tikzpicture}

\draw[thick] (0,0) -- (-1,-1); 
\draw[thick] (0,0) -- (0,-1); 
\draw[thick] (0,0) -- (1,-1); 

\draw[thick] (-1,-1) -- (-2,-2); 
\draw[thick] (-1,-1) -- (-1,-2); 
\draw[thick] (-1,-1) -- (0,-2); 

\node at (-1,-2.3){$A$};
\node at (0,-2.3){$E$};
\node at (0,-1.3){$x$};
\node at (1,-1.3){$G$};

\filldraw[thick](0,0) circle (.04);
\filldraw[thick](-1,-1) circle (.04);
\filldraw[thick](0,-1) circle (.04);
\filldraw[thick](1,-1) circle (.04);
\filldraw[thick](-2,-2) circle (.04);
\filldraw[thick](-1,-2) circle (.04);
\filldraw[thick](0,-2) circle (.04);

\end{tikzpicture}
\end{center}
\caption{The leaves of the given tree determine the basepoint
$\omega_{n,\ell}$.}
\label{mytree}
\end{figure}

Now we read the word labelling the leaves of the tree, $\omega_{3,\ell} = AExG$. This is the baseword of our diagram group. We then claim that there is an isomorphism:
\[ G_{3,\ell} \cong \mathrm{D}(\mathcal{P}_{3,\ell}, \omega_{3,\ell}). \]
We shall now turn to justifying the general procedure sketched above. 

\section{Local generation of the groups $G_{n,(a,b)}$} \label{section:local}

Many generalized Thompson groups $G$ have piecewise definitions: for a given $h \in G$, there is a finite subdivision of $Dom(h)$ into pieces $P_{1}, \ldots, P_{m}$ such that the restriction of $h$ to $P_{i}$ is a function having certain additional properties. 
A central idea in \cite{FH} and \cite{Far} is to collect these smaller functions
into an object called an inverse semigroup $S$, and then to view $G$ as being ``locally generated" (or locally determined) by the inverse semigroup $S$. Both \cite{FH} and \cite{Far}
also use the semigroup $S$ as the starting point for a construction of a space on which $G$ acts. This is the general pattern to be followed in the remainder of the paper.

In this section and the next, we will carefully describe inverse semigroups $S_{n,\ell}$ that ``locally generate" $G_{n,\ell}$. The analysis in the current section is confined to the groups $G_{n,(a,b)}$ (i.e., $k=2$). Some of the material here overlaps with standard accounts such as \cite{CFP}, which covers the case $(n,a,b) = (2,0,1)$.

\subsection{Inverse semigroups and associated groups}

\begin{definition}(Inverse semigroups; the action of an inverse semigroup on a set) \label{definition:inv}
Let $X$ be a set. A \emph{partial bijection} of $X$ is a bijective function $h: A \rightarrow B$ between subsets $A$, $B$ of $X$. The function with empty domain and range will be denoted by $0$. 

An \emph{inverse semigroup} $S$ is a collection of partial bijections that is closed under compositions and inverses, where compositions are defined ``on overlaps": if $s_{1}, s_{2} \in S$, then the domain of $s_{1} \circ s_{2}$ is $s_{2}^{-1}(Dom(s_{1}))$. 
If $S$ is an inverse semigroup of partial bijections of $X$, then we say that $S$ \emph{acts} on $X$.    
\end{definition}

\begin{definition}\cite{FH, Far}(The group $\Gamma(S)$)
\label{definition:Ginv}
Let $S$ be an inverse semigroup acting on a set $X$. A bijection $h: X \rightarrow X$
is \emph{locally determined} by $S$ if, for some $m$, there are partitions
$\mathcal{P}_{1} = \{ A_{1}, \ldots, A_{m} \}$ and $\mathcal{P}_{2} = \{ B_{1}, \ldots, B_{m} \}$ of $X$ and partial bijections $s_{i}: A_{i} \rightarrow B_{i}$ (for $i = 1, \ldots, m$) such that $h_{\mid A_{i}} = s_{i}$ and $s_{i} \in S$ for all $i$. We sometimes denote such an $h$ by
$h(s_{1}, \ldots, s_{m})$.

We let $\Gamma(S)$ be the collection of all homeomorphisms of $X$ that are locally determined by 
$S$.
\end{definition}

\begin{remark}
The set $\Gamma(S)$ becomes a group under composition of functions if $\Gamma(S) \neq \emptyset$. The latter is true precisely when $X$ can be partitioned by the domains of some collection $\{ s_{1}, \ldots, s_{n} \} \subseteq S$.  
\end{remark}

\subsection{Standard $n$-adic intervals and standard $n$-adic transformations}

\begin{definition}
\label{definition:nadic}
($N$-adic numbers, standard $n$-adic intervals; standard $n$-adic transformations; $S_{n}$)
If $a$ and $m$ are non-negative integers, then $a/n^{m}$ is an \emph{$n$-adic} number.  The interval
\[ I = \left[ \frac{a}{n^{m}}, \frac{a+1}{n^{m}} \right)  \]
is a \emph{standard $n$-adic interval} if $I \subseteq [0,1]$. 

If $I_{1}$ and $I_{2}$ are standard $n$-adic intervals, then we let $\sigma_{I_{1}}^{I_{2}}$ be the unique orientation-preserving linear homeomorphism from $I_{1}$ to $I_{2}$. We say that $\sigma_{I_{1}}^{I_{2}}$ is a \emph{standard $n$-adic transformation}, and we let $S_{n}$ be the set of all $n$-adic transformations. 
\end{definition}



\begin{definition}(The address of a standard $n$-adic interval) 
Let $\omega \in \{ 0, \ldots, n-1 \}^{\ast}$. The set of all real 
numbers $x$ such that $x$ admits an $n$-ary expansion
\[ .\omega a_{1} a_{2} a_{3} \ldots \]
is the closure of a standard $n$-adic interval. If we omit the number
$.\omega \overline{n-1}$,
then the remaining $x$ form a standard $n$-adic interval, which we denote $I_{\omega}$. 

Conversely, if a standard $n$-adic interval $I$ is given, then there is a unique $\omega$ such that $I = I_{\omega}$. We sometimes refer to $\omega$ as the \emph{address} of $I$ in this circumstance. 

The assignment $\omega \mapsto I_{\omega}$ determines a bijection
between $\{ 0, \ldots, n-1 \}^{\ast}$ and the collection of all standard $n$-adic intervals.
\end{definition}

\begin{convention}
In view of the bijection between addresses and standard $n$-adic intervals, we will write 
$\sigma^{\omega_{2}}_{\omega_{1}}$ in place of $\sigma^{I_{\omega_{2}}}_{I_{\omega_{1}}}$. 
\end{convention}

\begin{lemma} (Basic properties of the transformations
$\sigma_{\omega_{1}}^{\omega_{2}}$)  \label{lemma:basic}
Let $\omega_{1}, \omega_{2}$ be words in the alphabet $\{ 0, \ldots, n-1 \}$.
\begin{enumerate}
\item The standard $n$-adic transformation $\sigma_{\omega_{1}}^{\omega_{2}}$ is defined by a rule of the form
$x \mapsto ax + b$,
where $a$ is a power of $n$ and $b$ is an $n$-adic number in $[0,1)$.
\item The transformation $\sigma_{\omega_{1}}^{\omega_{2}}$ acts by prefix replacement. That is, if $x$ admits an $n$-ary expansion of the form $.\omega_{1} a_{1} a_{2} \ldots$
in which not all of the $a_{i}$ are $n-1$, then 
\[ \sigma_{\omega_{1}}^{\omega_{2}} (x) = .\omega_{2} a_{1} a_{2} \ldots. \]
The unique continuous extension of $\sigma_{\omega_{1}}^{\omega_{2}}$ to the closure of $I_{\omega_{1}}$ acts by prefix replacement on $.\omega_{1}\overline{n-1}$ as well.
\item The transformation $\sigma_{\omega_{1}}^{\omega_{2}}$ fixes a unique $x$ in 
$I_{\omega_{1}}$ if and only if 
$\omega_{i} \equiv \omega_{j} \tau$ ($\{i,j \} = \{ 1, 2 \}$) for some $\tau \in \{ 0, \ldots, n-1 \}^{+} - \{ n-1 \}^{+}$. If the latter holds, then the fixed point
is the rational number $.\omega_{j} \overline{\tau}$.
\item The transformation $\sigma_{\omega_{1}}^{\omega_{2}}$ has multiple fixed points if and only if $\omega_{1} \equiv \omega_{2}$, and thus $\sigma_{\omega_{1}}^{\omega_{2}} = id_{\omega_{1}}$. The same is true of the continuous extension of $\sigma_{\omega_{1}}^{\omega_{2}}$.
\end{enumerate}
\end{lemma}

\begin{proof}
We prove (3), leaving (1), (2) and (4) to the reader.

Assume first that $\sigma_{\omega_{1}}^{\omega_{2}}$ has a unique fixed point $x$ in $I_{\omega_{1}}$. We assume that 
$.b_{1}b_{2} \ldots$ is the unique $n$-ary expansion of $x$ that does not end in an infinite sequence of $n-1$s. Since $x \in I_{\omega_{1}} \cap I_{\omega_{2}}$, both $\omega_{1}$ and $\omega_{2}$ are prefixes of
$b_{1}b_{2}b_{3} \ldots$. If $\omega_{1} \equiv \omega_{2}$, then 
the transformation $\sigma_{\omega_{1}}^{\omega_{2}}$ is the identity on
$I_{\omega_{1}}$, and thus there are infinitely many fixed points, contrary to our hypothesis. It must therefore be that $\omega_{1} \equiv \omega_{2} \tau$ or $\omega_{2} \equiv \omega_{1} \tau$, for some $\tau \in \{ 0, \ldots, n-1 \}^{+}$.
Suppose, for a contradiction, that $\tau \in \{ n-1 \}^{+}$ and, without loss of generality, that $\omega_{2} \equiv \omega_{1} \tau$. We apply the $m$th power of $\sigma_{\omega_{1}}^{\omega_{2}}$ to $x$, assuming that $x = .\omega_{1} b_{\ell} b_{\ell+1} \ldots$ is the $n$-ary expansion of $x$ that does not end in an infinite sequence of $n-1$s:
\begin{align*}
.\omega_{1} b_{\ell} b_{\ell+1} \ldots &= \left( \sigma_{\omega_{1}}^{\omega_{2}} \right)^{m}(.\omega_{1}b_{\ell}b_{\ell+1} \ldots) \\
&= .\omega_{1} \tau^{m} b_{\ell} b_{\ell+1}\ldots.
\end{align*}
Comparing these $n$-ary expansions, and letting $m$ be arbitrarily large, we find that $b_{j} = n-1$ for arbitrarily large $j$. This is a contradiction. Thus,
$\tau \not \in \{ n-1 \}^{+}$, which proves the forward direction.

Conversely, assume without loss of generality that $\omega_{1} \equiv \omega_{2} \tau$, where $\tau \in \{ 0, \ldots, n-1 \}^{+} - \{ n-1 \}^{+}$. Consider the
number $x$ with the $n$-ary expansion $.\omega_{2} \overline{\tau}$. We apply $\sigma_{\omega_{1}}^{\omega_{2}}$ to $x$, using (2):
\begin{align*}
\sigma_{\omega_{1}}^{\omega_{2}}(x) &= \sigma_{\omega_{1}}^{\omega_{2}}(.\omega_{2}\tau \tau \tau \ldots) \\
&= \sigma_{\omega_{1}}^{\omega_{2}}(.\omega_{1} \tau \tau \ldots) \\
&= .\omega_{2} \tau \tau \tau \ldots.
\end{align*}
It follows that $.\omega_{2} \overline{\tau}$ is a fixed point of
$\sigma_{\omega_{1}}^{\omega_{2}}$. If $z=.b_{1}b_{2}b_{3}$ is an arbitrary fixed point, where infinitely many of the $b_{i}$ are not equal to $n-1$, then $\omega_{1}$ is a prefix of $z$. We can therefore 
write $\omega_{2} \equiv b_{1}b_{2}\ldots b_{\alpha}$, $\omega_{1} \equiv b_{1}b_{2} \ldots b_{\beta}$, and $\tau \equiv b_{\alpha + 1} \ldots
b_{\beta}$, for some $\alpha < \beta$. Applying $\sigma_{\omega_{1}}^{\omega_{2}}$, we find that
\[ .b_{1}b_{2}b_{3} \ldots \, \, = \, \, .b_{1} \ldots b_{\alpha} b_{\beta + 1} b_{\beta + 2} \ldots. \]
Comparing these $n$-ary expansions, we conclude that 
$b_{\alpha + m} = b_{\beta + m}$ for all $m \in \mathbb{N}$. We easily conclude that $z = .\omega_{2} \overline{\tau}$, so the fixed point of
$\sigma_{\omega_{1}}^{\omega_{2}}$ is unique.  
\end{proof}

\subsection{Local generation of the groups $G_{n,(a,b)}$}
\label{subsection:Gnab}

\begin{lemma} (Subdivision into standard $n$-adic transformations) \label{lemma:subdivision}
Let $f: [a,b) \rightarrow [c,d)$ be an increasing linear homeomorphism such that
\begin{enumerate}
\item $a$, $b$, $c$, and $d$ are $n$-adic numbers in $[0,1]$, and 
\item $f'(x)$ is a fixed integral power of $n$, for all $x \in [a,b)$.
\end{enumerate}
There are partitions $\{ I_{1}, \ldots, I_{m} \}$ and $\{ I'_{1}, \ldots, I'_{m} \}$ of $[a,b)$ and $[c,d)$ (respectively) into standard $n$-adic intervals such that, for each $j$, $f_{\mid I_{j}} : I_{j} \rightarrow I'_{j}$ is a standard $n$-adic transformation. 
\end{lemma}

\begin{proof}
We can first write $a$ and $b$ over a common power of $n$:
\[ a = \frac{\alpha}{n^{k}}; \quad \quad b = \frac{\beta}{n^{k}}. \]
Subdivide $[a,b)$ into the intervals
\[ \left[ \frac{\alpha}{n^{k}}, \frac{\alpha+1}{n^{k}} \right), \quad
\left[ \frac{\alpha+1}{n^{k}}, \frac{\alpha+2}{n^{k}} \right), \quad \ldots \quad 
\left[ \frac{\beta - 1}{n^{k}}, \frac{\beta}{n^{k}} \right).
\] Each of these intervals is a standard $n$-adic interval.
The conditions on $f$ imply that $f(x) = n^{\alpha}x+d$, where $\alpha \in \mathbb{Z}$ and $d \in \mathbb{Z}[1/n]$. It follows easily that the restriction of 
$f$ to each subinterval also satisfies the hypotheses of Lemma \ref{lemma:subdivision}. It therefore suffices to prove the lemma in the case that $[a,b)$ is a standard $n$-adic interval.

Since $b-a$ is a power of $n$ and $f'(x)$ is also a power of $n$, $d-c$ is a power of $n$; $d-c = n^{m}$, say. Let 
\[ c = \frac{\gamma}{n^{\ell}} \quad \text{and} 
\quad d = \frac{\delta}{n^{\ell}}, \]
where $\gamma, \delta \in \mathbb{N} \cup \{ 0 \}$.
A direct calculation shows that $\delta - \gamma = n^{\ell + m}$. It follows that $\ell +m$ is a non-negative integer. We subdivide $[c,d)$ into intervals as follows:
\[ \left[ \frac{\gamma}{n^{\ell}}, \frac{\gamma+1}{n^{\ell}} \right), \quad
\left[ \frac{\gamma+1}{n^{\ell}}, \frac{\gamma+2}{n^{\ell}} \right), \quad \ldots \quad 
\left[ \frac{\delta - 1}{n^{\ell}}, \frac{\delta}{n^{\ell}} \right).
\]
We note that there are precisely $n^{\ell+m}$ of these intervals. If we now subdivide the domain $[a,b)$ into $n^{\ell+m}$ equal parts, then each of these parts remains a standard $n$-adic interval, and the restrictions of $f$ map each of these parts to one of the above intervals by a standard $n$-adic transformation. 
\end{proof}

\begin{lemma} (Local generation of $G_{n,(a,b)}$ by standard $n$-adic transformations)
\label{lemma:localgen}
Let $[a,b] \subseteq [0,1]$ and $n > 1$. If $h \in G_{n,(a,b)}$, then there are 
standard $n$-adic transformations $\sigma_{1}, \ldots, \sigma_{\ell}$ such that 
the restriction of $\gamma(\sigma_{1}, \ldots, \sigma_{\ell})$ to $[a,b]$ is equal to $h$.
\end{lemma}

\begin{proof}
This is entirely straightforward if $h$ has no singularities, so we assume that $h$ has a non-empty collection of singularities $a < s_{1} < s_{2} < \ldots < s_{\delta} < b$. Consider the (unique) piecewise linear extension $\hat{h}: \mathbb{R} \rightarrow \mathbb{R}$ of $h$ such that $\hat{h}$ and $h$ share the same set of singularities. By continuity of $\hat{h}$, we can find $s_{0} \in \mathbb{Z}[1/n] \cap [0,a]$
and $s_{\delta+1} \in \mathbb{Z}[1/n] \cap [b,1]$ such that 
$\hat{h}(s_{0}), \hat{h}(s_{\delta+1}) \in [0,1]$. 

We apply Lemma \ref{lemma:subdivision} to each interval $[s_{i},s_{i+1})$ ($i = 0, \ldots, \delta$), and thus find 
a sequence of standard $n$-adic transformations $\sigma_{1}, \ldots, \sigma_{\ell}$ such that $\gamma := \gamma(\sigma_{1}, \ldots, \sigma_{\ell}) =
\hat{h}_{\mid [s_{0},s_{\delta+1})}$. Clearly, $\gamma$ has the desired property.  
\end{proof}

\section{Local generation of the groups $G_{n,\ell}$}
\label{section:genlocal} 

Throughout this section, we write $\ell = (r_{1}, \ldots, r_{k})$ and we let
$A_{n}(r_{i})$ denote the $n$-ary automaton for $r_{i}$, as described in Section \ref{section:example}. We also let $T_{n}$ denote the infinite rooted $n$-ary tree.

\subsection{The labelled tree $T_{n,\ell}$}

\begin{definition} ($\omega$-accepting state) \label{definition:accept}
Let $\omega$ be a word that is accepted by $A_{n}(x)$. If $p$ is the path from the basepoint of $A_{n}(x)$ that is labelled by $\omega$, then we say that $\tau(p)$ is the \emph{$\omega$-accepting state} of $A_{n}(x)$.
\end{definition}

\begin{definition} (left and right failure) \label{definition:leftright}
Let $\omega = a_{1}\ldots a_{m}$ ($a_{i} \in \{ 0, \ldots, n-1 \}$).
If $\omega$ is not accepted by $A_{n}(x)$, then there is a (possibly empty) maximal proper prefix $\omega' = a_{1}\ldots a_{\beta}$  such that $\omega'$ is accepted by $A_{n}(x)$.
\begin{enumerate}
\item Assume $\omega$ is not accepted by $A_{n}(r_{1})$ and $\omega'$ is as above.
We say that $\omega$ \emph{left fails} if $a_{\beta + 1}$ is less than the label of the outgoing edge from the $\omega'$-accepting state of $A_{n}(r_{1})$;
\item Assume $\omega$ is not accepted by $A_{n}(r_{k})$ and $\omega'$ is as above.
We say that  $\omega$ \emph{right fails} if $a_{\beta + 1}$ is greater than the label of the outgoing edge from the $\omega'$-accepting state of $A_{n}(r_{k})$.
\end{enumerate} 
\end{definition}

\begin{lemma} \label{lemma:labellings}
(Interpretation of acceptance and failure)
Let $\omega$ be a node in $T_{n}$ and let $I_{\omega}$ be the standard $n$-adic interval whose address is $\omega$.
\begin{enumerate}
\item $\omega$ left fails if and only if, for every $x \in I_{\omega}$, $x < r_{1}$;
\item $\omega$ right fails if and only if, for every $x \in I_{\omega}$, $x \geq r_{k}$;
\item $\omega$ is accepted by $A_{n}(r_{i})$ ($i \in \{ 1, \ldots, k-1 \}$)
if and only if $r_{i} \in I_{\omega}$;
\item if no $n$-ary expansion of $r_{k}$ is terminating, then $\omega$ is accepted by $A_{n}(r_{k})$ if and only if $r_{k} \in I_{\omega}$;
\item if $r_{k}$ admits a terminating expansion, then $\omega$ is accepted by $A_{n}(r_{k})$ if and only if $r_{k}$ is in the interior of $I_{\omega}$ or $r_{k}$ is the right endpoint of $I_{\omega}$. (Here $b$ is the right endpoint of  $[a,b)$).
\end{enumerate}
\end{lemma}

\begin{proof}
All parts of the lemma follow from elementary facts about the $n$-ary expansion of a real number. We leave the details to the reader.
\end{proof}

\begin{lemma} \label{lemma:stab}(Stabilizers  of the points $r_{i}$ in $S_{n}$)
Let $\sigma_{\omega_{1}}^{\omega_{2}} \in S_{n}$. 
\begin{enumerate}
\item If $i<k$ or if $i=k$ and $r_{k}$ has no terminating $n$-ary expansion, then  
$\sigma_{\omega_{1}}^{\omega_{2}}(r_{i}) = r_{i}$ if and only if $\omega_{1}$ and
$\omega_{2}$ are accepted by $A_{n}(r_{i})$, and have the same accepting states. 
\item If $r_{k}$ has a terminating $n$-ary expansion and neither $\omega_{1}$ nor $\omega_{2}$ right fail, then the unique continuous 
extension of $\sigma_{\omega_{1}}^{\omega_{2}}$ to the closure of $\omega_{1}$ fixes $r_{k}$ if and only if
$\omega_{1}$ and $\omega_{2}$ are accepted by $A_{n}(r_{k})$ and have the same accepting states.
\end{enumerate}
\end{lemma}

\begin{proof}
We first prove (1). Under the given hypotheses, the unique $n$-ary expansion 
of $r_{i}$ that does not end in an infinite sequence of $n-1$s is also the expansion that labels the automaton $A_{n}(r_{i})$. We consider two cases.

Assume first that $r_{i}$ is irrational. In this case, the automaton
$A_{n}(r_{i})$ is a labelled directed ray. Assume that $\sigma_{\omega_{1}}^{\omega_{2}}(r_{i}) = r_{i}$. By Lemma \ref{lemma:basic}(3), if $r_{i}$ were the unique fixed point of
$\sigma_{\omega_{1}}^{\omega_{2}}$, then $r_{i}$ would be rational. Thus, $\sigma_{\omega_{1}}^{\omega_{2}}$ has multiple fixed points.
Lemma \ref{lemma:basic}(4) now implies that $\omega_{1} \equiv \omega_{2}$. The words $\omega_{1}$ and $\omega_{2}$ are both accepted by $A_{n}(r_{i})$ (by Lemma \ref{lemma:labellings}(3-4)) and clearly they must have the same accepting state.

Conversely, if both $\omega_{1}$ and $\omega_{2}$ are accepted by $A_{n}(r_{i})$ and they have the same 
accepting state, then $\omega_{1} \equiv \omega_{2}$ (since $A_{n}(r_{i})$ is a ray) and one easily concludes that $\sigma_{\omega_{1}}^{\omega_{2}}$ fixes $r_{i}$. This handles the case in which $r_{i}$ is irrational.

Assume that $r_{i}$ is rational and $\sigma_{\omega_{1}}^{\omega_{2}}$ fixes $r_{i}$. It follows that $r_{i} \in I_{\omega_{1}} \cap I_{\omega_{2}}$, so 
$A_{n}(r_{i})$ accepts both $\omega_{1}$ and $\omega_{2}$ by Lemma \ref{lemma:labellings}(3-4). If $\omega_{1} \equiv \omega_{2}$, then $\omega_{1}$ and $\omega_{2}$ clearly have the same accepting states. It therefore suffices to consider the case in which
$\omega_{1} \not \equiv \omega_{2}$. The transformation $\sigma_{\omega_{1}}^{\omega_{2}}$ has a unique fixed point, by our current hypothesis and Lemma \ref{lemma:basic}(4). We conclude from
Lemma \ref{lemma:basic}(3) that either $\omega_{1} \equiv \omega_{2} \tau$ or $\omega_{2} \equiv \omega_{1} \tau$, where $\tau \in \{ 0, \ldots , n-1 \}^{+} - \{ n-1 \}^{+}$. We assume the former, the latter case being similar. Thus, $\omega_{1} \equiv \omega_{2} \tau$, and
\[ r_{i} = .\omega_{2}\overline{\tau}, \]
since $.\omega_{2}\overline{\tau}$ and $r_{i}$ are both fixed points
of $\sigma_{\omega_{1}}^{\omega_{2}}$.   
Let $\alpha$ be the word in $\{ 0, \ldots, n-1 \}^{\ast}$ that labels the path 
in $A_{n}(r_{i})$ from $v$ to $v'$, and let $\beta$ label the directed loop of
$A_{n}(r_{i})$, starting at $v'$. (Note that $v$ and $v'$ are the vertices described in Section \ref{section:example}.) We have
\[ .\alpha \overline{\beta} = .\omega_{2} \overline{\tau}, \]
where the digits after the decimal point agree place by place. It now follows from Remark \ref{remark:minimal} that the $\omega_{2}$-accepting state $v''$ in $A_{n}(r_{i})$ lies on the loop of $A_{n}(r_{i})$, and that $|\tau| = m |\beta|$, for some $m>0$. This implies that the $\omega_{1}$-accepting state $v'''$ is obtained from $v''$ by tracing the loop of  $A_{n}(r_{i})$ $m$ complete times, beginning at $v''$. Thus, we have $v''' =v''$, as desired.

Conversely, if $\omega_{1}$ and $\omega_{2}$ are both accepted by $A_{n}(r_{i})$ and have the same accepting state $v''$, then $\omega_{1} \equiv \omega_{2} \tau$ or $\omega_{2} \equiv \omega_{1} \tau$, where we may assume that $\tau$ is the label of a non-trivial directed loop in $A_{n}(r_{i})$ at $v''$. (The case in which $\tau$ is a trivial word is handled easily.) We assume, without loss of generality, that
$\omega_{1} \equiv \omega_{2} \tau$, where $\tau \in \{ 0, 1, 2, \ldots, n-1 \}^{+}$. Since $v''$ can be connected to itself by a non-trivial directed loop, it must be that $v''$ lies on the loop of $A_{n}(r_{i})$. If we let $\beta$ denote the label of the simple directed loop from $v'$ to itelf in $A_{n}(r_{i})$, then we must have that 
$\tau$ is a cyclic shift of some positive power $m$ of $\beta$. In particular, $\tau \not \in \{ n-1 \}^{+}$. Lemma \ref{lemma:basic}(3) now implies that $.\omega_{2} \overline{\tau}$ is the unique fixed point of $\sigma_{\omega_{1}}^{\omega_{2}}$. We let $\alpha$ denote the label of the directed arc in $A_{n}(r_{i})$ from $v$ to $v'$. Since the accepting state $v''$ of $\omega_{2}$ lies on the loop of
$A_{n}(r_{i})$, we find that $\omega_{2} \equiv \alpha \beta^{m} \gamma$ (for some $m \geq 0$), where $\gamma$ labels the directed arc from $v'$ to $v''$. Writing $\tau \equiv (\delta \gamma)^{m}$, where $\delta$ labels the directed arc from $v''$ to $v'$, and comparing $r_{i} = .\alpha \overline{\beta}$ with $.\omega_{2} \overline{\tau}$, we find that the latter is $r_{i}$, as required. 

Now (2). We note that $r_{k} = .\alpha \overline{n-1}$, where
$\alpha$ labels the path in $A_{n}(r_{k})$ from $v$ to $v'$ (and therefore does not end in $n-1$) and $v'$ is the only vertex 
on the loop of $A_{n}(r_{k})$. The unique continuous extension of $\sigma_{\omega_{1}}^{\omega_{2}}$ will be denoted by $\widehat{\sigma}_{\omega_{1}}^{\omega_{2}}$. We note that, by Lemma \ref{lemma:labellings}(2), $r_{k}$ cannot be the left endpoint of either $I_{\omega_{1}}$ or $I_{\omega_{2}}$. 

Assume that $\widehat{\sigma}_{\omega_{1}}^{\omega_{2}}$ fixes $r_{k}$. It follows from Lemma \ref{lemma:labellings}(5) that $\omega_{1}$ and $\omega_{2}$ are accepted by $A_{n}(r_{k})$. We assume, without loss of generality, that $\omega_{1} \equiv \omega_{2} \tau$ for some $\tau \in \{ 0, \ldots, n-1 \}^{+}$, the case $\omega_{1} \equiv \omega_{2}$ being trivial. 
A direct calculation shows that
$.\omega_{2}\overline{\tau}$ is a fixed point of $\widehat{\sigma}_{\omega_{1}}^{\omega_{2}}$. By Lemma \ref{lemma:basic}(4), if $\widehat{\sigma}_{\omega_{1}}^{\omega_{2}}$ had multiple fixed points, then we would have
$\omega_{1} \equiv \omega_{2}$, contrary to our hypothesis. Thus,
$.\alpha \overline{n-1} = .\omega_{2}\overline{\tau}$. By minimality of
$\alpha$, we have that $\omega_{2}$ is at least as long as $\alpha$. The same is true of $\omega_{1}$, since $|\omega_{1}| > |\omega_{2}|$. This means that both 
$\omega_{1}$ and $\omega_{2}$ have the accepting state $v'$. This proves the forward direction.

Conversely, suppose that $\omega_{1}$ and $\omega_{2}$ are accepted by 
$A_{n}(r_{k})$ and that both have the same accepting state. By Lemma \ref{lemma:labellings}(5), $r_{k}$ is in the domain of 
$\widehat{\sigma}_{\omega_{1}}^{\omega_{2}}$. Since the accepting states of
$\omega_{1}$ and $\omega_{2}$ are identical, we can assume (without loss of generality) that $\omega_{2} \equiv \omega_{1} \tau$ and find a directed loop $p$ from the accepting state $v'''$ of $\omega_{1}$ to the accepting state $v''$ of $\omega_{2}$. If $v'''$ is not on the loop 
of $A_{n}(r_{k})$, then this is only possible if $p$ is trivial. If $p$ is trivial, we have
$\omega_{1} \equiv \omega_{2}$, and so $\widehat{\sigma}_{\omega_{1}}^{\omega_{2}}(r_{k}) = r_{k}$, as desired (Lemma \ref{lemma:basic}(4)).
We therefore can assume that $v'''$ is on the loop of $A_{n}(r_{k})$ and that $p$ is non-trivial. It follows that 
$\tau \equiv (n-1)^{m}$ for some $m > 0$.
 By Lemma \ref{lemma:basic}(2), $\widehat{\sigma}_{\omega_{1}}^{\omega_{2}}$ acts by prefix replacement
on $.\omega_{1} \overline{\tau}$. A straightforward calculation then shows that
$\widehat{\sigma}_{\omega_{1}}^{\omega_{2}}$ fixes $.\omega_{1} \overline{\tau}$. However, by minimality of $\alpha$, we have
$\omega_{1} \equiv \alpha (n-1)^{j}$ for some $j$, so $r_{i} = .\alpha \overline{n-1}= 
\omega_{1} \overline{\tau}$ is fixed, as desired.  

\end{proof}

\begin{definition} \label{definition:Tnl}
(The labelled tree $T_{n,\ell}$; reduced vertices in $T_{n,\ell}$)
Let $T_{n}$ be the rooted ordered infinite $n$-ary tree. We will label certain nodes of $T_{n}$ in the following way.
Let $v \in T_{n}$ be a node, and let $\omega$ denote the address of $v$.
\begin{enumerate}
\item If $\omega$ is accepted by a unique automaton $A_{n}(r_{i})$, then we assign $v$ the label of the $\omega$-accepting state of $A_{n}(r_{i})$;
\item If $\omega$ is accepted by multiple automata $A_{n}(r_{i})$, then we assign $v$ no label;
\item If $\omega$ left or right fails, then we also assign $v$ no label;
\item In all other cases, $v$ is assigned the label $x$.
\end{enumerate}
Let $T_{n,\ell}$ denote the $n$-ary tree with the above labelling.

A node $v$ of $T_{n,\ell}$ is \emph{reduced} if it is labelled and at least two of its children are also labelled.
\end{definition}

\begin{lemma} \label{lemma:+} 
(Labelled isomorphism of subtrees in $T_{n,\ell}$)
If $v_{1}$ and $v_{2}$ are reduced nodes in $T_{n,\ell}$ with the same labels, then 
the subtrees rooted at $v_{1}$ and $v_{2}$ are isomorphic as labelled ordered rooted trees. The isomorphism and its inverse both carry reduced nodes to reduced nodes.
\end{lemma}

\begin{proof}
Assume that $v_{1}$ and $v_{2}$ are reduced nodes in $T_{n,\ell}$ with the same labels. If these labels are $x$, then the subtrees rooted at $v_{1}$ and $v_{2}$ 
are infinite $n$-ary trees in which all nodes are labelled by $x$, and in which all nodes are reduced. The desired conclusion is then immediate.

Otherwise, the addresses of the nodes $v_{1}$ and $v_{2}$ have identical accepting states in one of the automata $A_{n}(r_{i})$. Since the labellings of the children and the status of the children as reduced or unreduced is purely a function of this accepting state, the desired conclusion follows by induction. 
\end{proof}

\begin{example} (The tree $T_{3,(5/26, 7/26, 5/6)}$)
We briefly discuss the labelled tree $T_{n,\ell}$, where $n$ and $\ell$ are as from Section \ref{section:example}.
Since $n=3$, $T_{n,\ell}$ is an infinite ternary tree. The root $\varepsilon$ is assigned no label, since $\varepsilon$ is accepted by all three automata.
The node $0$ is also assigned no label, since both $A_{3}(5/26)$ and $A_{3}(7/26)$ accept $0$. The node $00$ (and all of its children) are assigned no label because $00$ left fails. The nodes $01$ and $02$ are labelled $A$ and $E$, respectively, by (1) from Definition \ref{definition:Tnl}. The node $1$ is labelled $x$. The node $2$ is labelled $G$.

If a given node $v$ has label $X$, then we write $X \rightarrow a_{1}a_{2}a_{3}$ to indicate that the children of $v$ are to be labelled $a_{1}$, $a_{2}$, and $a_{3}$ (from left to right). If a given node is to be unlabelled, we will, for the duration of this subsection, label that node by $\ast$. With this understanding, we have the following rules for labelling the children of \emph{reduced} nodes:
\begin{itemize}
\item $A \rightarrow Bxx$;
\item $B \rightarrow \ast Ax$;
\item $C \rightarrow Dxx$;
\item $D \rightarrow xxE$;
\item $E \rightarrow xCx$;
\item $F \rightarrow xxG$;
\item $G \rightarrow xG\ast$;
\item $x \rightarrow xxx$.
\end{itemize}
We note further that any node labelled by $\ast$ according to the above rules represents a node that left or right fails, and thus all of its children are also labelled $\ast$. 

The above tells us how to label the children of all nodes, except for the unreduced nodes, all of which are labelled $A$ in this example. Such a node follows the rule $A \rightarrow \ast \ast A$, where again all descendants of the first two children are also labelled $\ast$.
\end{example}

\begin{remark} \label{remark:connect} (The connection between $T_{n,\ell}$ and $\mathcal{P}_{n,\ell}$)
The labellings of the tree $T_{n,\ell}$ have a direct connection to the semigroup presentation 
$\mathcal{P}_{n,\ell}$ described in Section \ref{section:example}. For reduced labelled nodes $v$
of $T_{n,\ell}$, children are labelled by the rule $X \rightarrow \omega(X)$, where $X$ is the label of $n$ and $\omega(X)$ is (with one caveat) the right side of the ordered pair $(X,\omega(X))$ that was computed in Subsection \ref{subsection:labels}.  The caveat is that 
$\omega(X)$ is always padded to a total length of $n$, either by adding suitable $\ast$s to the left (if $X$ labels a state in $A_{n}(r_{1})$) or to the right (if $X$ labels a state in $A_{n}(r_{k})$. If $X$ labels an unreduced node $v$ (i.e., if the address of $v$ is accepted by an unreduced vertex of $A_{n}(r_{i})$ labelled $X$), then the rule 
is $X \rightarrow (\ast)^{n-1}X$ or $X \rightarrow X(\ast)^{n-1}$ (respectively). 
\end{remark}

\subsection{The inverse semigroups $S_{n,\ell}$}

\begin{definition} \label{definition:Snl}
($S_{n,\ell}$)
Let $S'_{n,\ell}$ denote the set containing
$0$ and all standard $n$-adic transformations $\sigma: I_{\omega_{1}} \rightarrow I_{\omega_{2}}$ such that $\omega_{1}$ and $\omega_{2}$ are the addresses of reduced  labelled nodes in $T_{n,\ell}$ with identical labels. We define
\[ S_{n,\ell} = \{ h: I'_{1} \rightarrow I'_{2} \mid h = \sigma_{\mid [r_{1},r_{k}]}, \text{ for some } \sigma \in S'_{n,\ell} \} \cup \{ 0 \}. \]
\end{definition}

\begin{remark}
We will continue to use the notation $\sigma_{\omega_{1}}^{\omega_{2}}$ to denote members of $S_{n,\ell}$. When we do this, $\sigma_{\omega_{1}}^{\omega_{2}}$ refers to its restriction to $[r_{1},r_{k}]$. 
\end{remark}

\begin{remark}
If $r_{k}$ has a terminating $n$-ary expansion, then (and only then) 
it will fail to be in the domain of any $h \in S_{n,\ell}$. In such cases,
the locally determined group $\Gamma(S_{n,\ell})$ will be a group of homeomorphisms of $[r_{1},r_{k})$. This is entirely harmless, since each 
$\gamma \in \Gamma(S_{n,\ell})$ will extend uniquely to a homeomorphism of $[r_{1},r_{k}]$. We can safely ignore this point in what follows.
\end{remark}

\begin{proposition} \label{proposition:Snl}
(The inverse semigroups $S_{n,\ell}$)
The set $S_{n,\ell}$ is an inverse semigroup with respect to the 
operation of composition.
\end{proposition}

\begin{proof}
It clearly suffices to check that $S'_{n,\ell}$ is an inverse semigroup.

We first show that $S'_{n,\ell}$ is closed under inverses. Let $\sigma_{\omega_{1}}^{\omega_{2}} \in S'_{n,\ell}$. This means that $\omega_{1}$ and $\omega_{2}$ are reduced labelled nodes, both of which have the same label. It follows directly that $\sigma_{\omega_{2}}^{\omega_{1}} \in S'_{n,\ell}$. Since 
\[ \left( \sigma_{\omega_{1}}^{\omega_{2}} \right)^{-1} = \sigma_{\omega_{2}}^{\omega_{1}}, \] $S'_{n,\ell}$ is closed under inverses. 


Let $\sigma_{\omega_{1}}^{\omega_{2}}$, $\sigma_{\omega_{3}}^{\omega_{4}} \in S'_{n,\ell}$. We suppose that the product $\sigma_{\omega_{3}}^{\omega_{4}} \sigma_{\omega_{1}}^{\omega_{2}}$ is not $0$. It follows that one of the strings $\omega_{2}$, $\omega_{3}$ is a prefix of the other. We assume that $\omega_{2}$ is a prefix of $\omega_{3}$, the reverse case being similar. Thus, $\omega_{3} \equiv \omega_{2} \tau$. Note that, by Lemma \ref{lemma:+}, $\omega_{2} \tau$ and $\omega_{1} \tau$ have the label, and $\omega_{1} \tau$ is a reduced node (since $\omega_{2}\tau$ is). We compute the product in $S_{n}$, noting that
$\sigma_{\omega_{1}\tau}^{\omega_{3}}$ is the restriction of $\sigma_{\omega_{1}}^{\omega_{2}}$
to $\omega_{1}\tau$:
\begin{align*}
\sigma_{\omega_{3}}^{\omega_{4}} \sigma_{\omega_{1}}^{\omega_{2}} &= \sigma_{\omega_{3}}^{\omega_{4}} \sigma_{\omega_{1}\tau}^{\omega_{3}} \\
&= \sigma_{\omega_{1}\tau}^{\omega_{4}}
\end{align*}
Since the nodes $\omega_{i}$, $i \in \{1, \ldots, 4 \}$, and $\omega_{1}\tau$ are reduced, and $\omega_{1}\tau$ and $\omega_{4}$ have the 
same label, 
$\sigma_{\omega_{3}}^{\omega_{4}} \sigma_{\omega_{1}}^{\omega_{2}} = \sigma_{\omega_{1}\tau}^{\omega_{4}} \in S'_{n,\ell}$.
\end{proof}

\subsection{Domains in $S_{n,\ell}$} 
\label{subsection:domains} 

The constructions of the spaces $\Delta(n,\ell)$ appeal to results
from \cite{FH} and \cite{Far}, which in turn use the domains of a given 
semigroup $S$ in an essential way. We therefore collect some properties 
of domains of $S_{n,\ell}$ here.

\begin{definition} (The set of non-empty domains)
\label{definition:domains}
We let 
\[ \mathcal{D}_{n,\ell}^{+} = \{ D \subseteq [0,1] \mid 
D \text{ is the domain of some } s \in S_{n, \ell} -\{0\} \}. \]
This is the \emph{set of non-empty domains}.
\end{definition}

\begin{lemma}(Domains and partitions)
\label{lemma:domandpar} 
If $v \in T_{n,\ell}$ is a labelled node with address $\omega$, then 
$v$ determines a unique domain $D \in \mathcal{D}_{n,\ell}^{+}$, namely,
$D = I_{\omega} \cap [r_{1},r_{k}]$. Conversely, 
if $D \in \mathcal{D}^{+}_{n,\ell}$, then $D$ determines a unique reduced labelled node, namely the domain of $id_{D} = \sigma_{D}^{D} \in S_{n,\ell}$.

If $v$ is any labelled node in $T_{n,\ell}$, then any partition of $v$ into domains is determined by a labelled cut set; i.e., a set $\{v_{1}, \ldots, v_{\alpha} \}$
such that 
\begin{enumerate}
\item each $v_{i}$ is labelled;
\item each descending ray issuing from $v$ that passes through only labelled vertices must pass through a unique $v_{i}$.
\end{enumerate}
\end{lemma}

\begin{proof}
This follows readily from our description of $S_{n,\ell}$ and from the fact that labelled nodes in $T_{n,\ell}$
correspond precisely to the standard $n$-adic intervals that meet $[r_{1}, r_{k}]$ in more than one point.

Note that the map from labelled nodes to domains is not one-to-one: each unreduced labelled node determines the same domain as its (unique) labelled child. The map is, however, bijective when restricted to reduced labelled nodes. 
\end{proof}

\subsection{Local generation of $G_{n,\ell}$}



\begin{theorem} (Local generation of $G_{n,\ell}$) \label{theorem:big}
\[ \Gamma(S_{n,\ell}) = G_{n,\ell}. \]
\end{theorem} 

\begin{proof}
Let $h \in \Gamma(S_{n,\ell})$. It follows that $h$ is a homeomorphism of the interval $[r_{1},r_{k}]$ and 
$h = (\sigma_{1}, \ldots, \sigma_{\beta})$, where $\sigma_{1}, \ldots, \sigma_{\beta}
\in S_{n,\ell}$. The derivatives of the $\sigma_{i}$ are in 
$\langle 1/n \rangle$ by definition. Singularities occur only at endpoints of the $\sigma_{i}$, all of which lie either at endpoints of $[r_{1},r_{k}]$ or at members of $\mathbb{Z}[1/n]$. Each $\sigma_{i}: I_{1} \rightarrow I_{2}$ is easily seen to determine a bijection between
$I_{1} \cap \mathbb{Z}[1/n]$ and $I_{2} \cap \mathbb{Z}[1/n]$. It thus suffices to show that $h$ fixes the points $r_{2}, \ldots, r_{k-1}$. Consider an arbitrary
$r_{i}$ $(2 \leq i \leq k-1)$. Since $r_{i}$ is in the domain 
of $h$, $r_{i}$ must be in the domain of $\sigma_{\gamma}$, for some $\gamma \in \{ 1, \ldots, \beta \}$. Suppose that $\sigma_{\gamma} = \sigma_{\omega_{1}}^{\omega_{2}}$. We note that $A_{n}(r_{i})$ must accept $\omega_{1}$ by Lemma \ref{lemma:labellings}(3).  Since $\omega_{1}$ and $\omega_{2}$ are both reduced and have the same label by the definition of $S_{n,\ell}$, Lemma \ref{lemma:stab}(1) implies that 
$\sigma_{\omega_{1}}^{\omega_{2}}(r_{i}) = r_{i}$. Since $r_{i}$ was arbitrary, 
$h \in G_{n,\ell}$.

Let $h \in G_{n,\ell}$. Since $h \in G_{n,[r_{1},r_{k}]}$, we can find $\sigma_{\alpha_{1}}^{\beta_{1}}, \ldots, \sigma_{\alpha_{d}}^{\beta_{d}} \in S_{n}$ such that the restriction of $\gamma(\sigma_{\alpha_{1}}^{\beta_{1}}, \ldots, \sigma_{\alpha_{d}}^{\beta_{d}})$ to 
$[r_{1},r_{k}]$ is $h$, by Lemma \ref{lemma:localgen}. After sufficient subdivision, we can further assume that the nodes $\alpha_{\delta}$ and $\beta_{\delta}$ are each accepted by at most one 
automaton $A_{n}(r_{i})$. 
By Lemma \ref{lemma:labellings}, we can further assume that each node
$\alpha_{1}, \ldots, \alpha_{d}$ is reduced and has a non-null label in $T_{n,\ell}$, by discarding the nodes (and corresponding transformations) whose labels left or right fail. It follows that the corresponding $\beta_{i}$ can neither left nor right fail, since $h$
preserves the interval $[r_{1},r_{k}]$. In particular, each $\beta_{i}$ necessarily has some non-null label 
in $T_{n,\ell}$. 

Let $i \in \{ 1, \ldots, d \}$ be arbitrary.
If $\sigma_{\alpha_{i}}^{\beta_{i}}$ fixes some 
$r_{j}$, then $\alpha_{i}$ and $\beta_{i}$ are accepted by $A_{n}(r_{j})$ and have the same accepting states, by Lemma \ref{lemma:stab}. Thus,
$\sigma_{\alpha_{i}}^{\beta_{i}} \in S_{n,\ell}$, as required. If $\sigma_{\alpha_{i}}^{\beta_{i}}$ does not fix any $r_{j}$, then none of the $r_{j}$
are in $I_{\alpha_{i}}$, and $r_{k}$ is neither in
$I_{\alpha_{i}}$ nor in its closure. 
It follows from this that $\alpha_{i}$ is not accepted by any of the automata $A_{n}(r_{j})$, and yet neither left nor right fails. Thus, $\alpha_{i}$ has the label
$x$. Now, if $r_{j} \in I_{\beta_{i}}$, then $h^{-1}(r_{j}) \neq r_{j}$, contrary to our hypothesis that $h \in G_{n,\ell}$. Thus,
$\beta_{i}$ also has the label $x$. Thus, $\sigma_{\alpha_{i}}^{\beta_{i}} \in S_{n,\ell}$, again as required.
\end{proof}

\section{CAT(0) cubical complex constructions for the groups $G_{n,\ell}$} \label{section:CAT(0)}

We are now ready to build a complex 
on which $G_{n,\ell}$ acts. Our construction begins with a simple expansion set $\mathcal{B}_{n,\ell}$.  Simple expansion sets \cite{Far2} are special types of the expansion sets defined in \cite{Far}. We will use the main result of  \cite{Far2} to prove that the simple expansion sets associated to $G_{n,\ell}$ give rise to CAT(0) cubical complexes $\Delta(n,\ell)$ (Theorem \ref{theorem:CAT(0)}).  We will then argue that
$G_{n,\ell}$ acts properly and freely by isometries on $\Delta(n, \ell)$ (Proposition \ref{proposition:action}). 

\subsection{The simple expansion set for $G_{n,\ell}$}

\begin{definition}  \cite{Far2} \label{definition:simpleexpansionsets}
(Simple expansion set)
A \emph{simple expansion set over $X$} is a $4$-tuple $(\mathcal{B}, X, supp, \mathcal{E})$, where $\mathcal{B}$ and $X$ are sets, and $supp: \mathcal{B} \rightarrow \mathcal{P}(X)$ and $\mathcal{E}$ are functions. 

For each $b \in \mathcal{B}$, $supp(b)$ is required to be a non-empty subset of $X$. The function $supp$ is called the \emph{support function}, and $supp(b)$ is the \emph{support} of $b$.  

A \emph{vertex} is a finite subset $v = \{ b_{1}, \ldots, b_{k}\} \subseteq \mathcal{B}$ such that $supp(b_{i}) \cap supp(b_{j}) = \emptyset$ when $i \neq j$. For each vertex $v$, we define 
\[ supp(v) = \bigcup_{\ell =1}^{k} supp(b_{\ell}); \quad \quad  
P(v) = \{ supp(b_{\ell}) \mid \ell \in \{ 1, \ldots, k \} \}.  \]
The collection $P(v)$ is the \emph{partition induced by $v$}. It is a partition of $supp(v)$.

 The function $\mathcal{E}$ assigns a set of vertices, denoted $\mathcal{E}(b)$, to each $b \in \mathcal{B}$. 
 The sets $\mathcal{E}(b)$ are required to satisfy the following three conditions:
\begin{enumerate}
\item $|\mathcal{E}(b)| \leq 2$;
\item $ \{ b \} \in \mathcal{E}(b)$;
\item If $|\mathcal{E}(b)| = 2$ and $v \in \mathcal{E}(b) - \{ \{ b \} \}$, 
then $P(v)$ is a proper refinement of $P(\{ b \})$. 
\end{enumerate}
A simple expansion set will usually be denoted by $\mathcal{B}$ (rather than by $(\mathcal{B},X,supp,\mathcal{E})$) for the sake of brevity. 
\end{definition}

\begin{definition} \label{definition:hatSnl} (The inverse semigroup $\widehat{S}_{n,\ell}$)
Let $h$ be a homeomorphism between two intervals, $I_{1}, I_{2} \subseteq [0,1]$. We say that $h$ is \emph{locally determined by $S_{n,\ell}$} if there are finitely many transformations $s_{1}, \ldots, s_{m} \in S_{n,\ell}$ with disjoint supports such that the domains of the $s_{i}$ partition the domain of $h$, and $h_{\mid Dom(s_{i})}
= s_{i}$ for $i = 1, \ldots, m$.

The set of all such homeomorphisms (for varying intervals $I_{1}$ and $I_{2}$, and varying $m$) is an inverse semigroup under composition if we include $0$. We denote this semigroup by $\widehat{S}_{n,\ell}$.
\end{definition}

\begin{definition} \cite{FH} \label{definition:B}(The set $\mathcal{B}_{n,\ell}$ and the function $supp$) Let
\[ \mathcal{A}_{n,\ell} = \{ (f,D) \mid f \in \widehat{S}_{n,\ell},
D \in \mathcal{D}^{+}_{n,\ell}, D \subseteq Dom(f) \}. \]
We define a relation $\sim$ on the set $\mathcal{A}_{n,\ell}$, writing
\[ (f_{1},D_{1}) \sim (f_{2},D_{2}) \]
if there is some bijection $s: D_{1} \rightarrow D_{2}$ such that $s \in S_{n,\ell}$ and $f_{2}s = f_{1}$ on $D_{1}$. The relation $\sim$ is an equivalence relation on $\mathcal{A}_{n,\ell}$. We denote the equivalence classes by $[f,D]$. The set of all such equivalence classes
is $\mathcal{B}_{n,\ell}$. 

We define $supp: \mathcal{B}_{n,\ell} \rightarrow \mathcal{P}([0,1])$
by the rule
\[ supp([f,D]) = f(D). \]
This is the \emph{support} of $[f,D]$. A straightforward check shows that the support is well-defined.
\end{definition}

\begin{definition} \label{definition:E}
(The function $\mathcal{E}$)
Let $[f,D] \in \mathcal{B}_{n,\ell}$. Since $D \in 
\mathcal{D}^{+}_{n,\ell}$, $D$ uniquely determines a reduced labelled node in $T_{n,\ell}$.
Let $D_{1}, \ldots, D_{\alpha}$ be the domains that correspond to labelled children of $D$ in $T_{n,\ell}$. We define 
\[ \mathcal{E}([f,D]) = \{ \{ [f,D] \}, \{ [f_{\mid D_{i}},D_{i}] \mid i = 1, \ldots, \alpha \} \}. \]  
For a given $b \in \mathcal{B}_{n,\ell}$, we let $E_{1}(b)$ denote the unique member of $\mathcal{E}(b) - \{ \{ b \} \}$ and let $E_{0}(b) = \{ b \}$.  
\end{definition}

\begin{remark} \label{remark:simply}
Thus, for a given $b = [f,D] \in \mathcal{B}$, $E_{0}(b)$ denotes the trivial expansion $\{ b \}$ and $E_{1}(b)$ denotes the expansion that splits $b$ into a number of pieces equal to the number of labelled children of the node determined by $D$. We have $\mathcal{E}(b) = \{ E_{0}(b), E_{1}(b) \}$. 
\end{remark}

\begin{proposition} \label{definition:Ewd}
($\mathcal{E}$ is well-defined)
The function $\mathcal{E}$ is well-defined.
\end{proposition}

\begin{proof}
Suppose we are given $[f_{1}, I_{\alpha}], [f_{2}, I_{\beta}]
\in \mathcal{B}_{n,\ell}$ such that $[f_{1}, I_{\alpha}] = [f_{2}, I_{\beta}]$. Thus, there is some $s \in S_{n, \ell}$ such that 
$s: I_{\alpha} \rightarrow I_{\beta}$ is a bijection, and such that $f_{2}s = f_{1}$ on $I_{\alpha}$. By the proof of Lemma \ref{lemma:domandpar}, we may assume that $\alpha$ and $\beta$ are reduced labelled nodes in $T_{n,\ell}$ with the same labels and $s = \sigma_{\alpha}^{\beta}$.
It is already clear that $E_{0}([f_{1},I_{\alpha}]) = E_{0}([f_{2},I_{\beta}])$. 

Now we would like to argue that $E_{1}([f_{1},I_{\alpha}]) = E_{1}([f_{2},I_{\beta}])$. For this, it suffices to show that 
$[f_{1}, I_{\alpha i}] = [f_{2},I_{\beta i}]$ for each labelled
child $\alpha i$ of $\alpha$. Here it is straightforward to argue that $\sigma_{\alpha i}^{\beta i}$ gives the desired transformation. Indeed, $\sigma_{\alpha i}^{\beta i}$ is the restriction of $\sigma_{\alpha}^{\beta}$ to $\alpha i$ and the nodes $\alpha i$ and $\beta i$ have the same labelling, by Lemma \ref{lemma:+}. If $\alpha i$ and $\beta i$ are not reduced, then simply replace both with their unique labelled children until reduced nodes $\alpha \tau$ and $\beta \tau$ are reached, and then observe that $\sigma_{\alpha i}^{\beta i} = \sigma_{\alpha \tau}^{\beta \tau}$.
\end{proof}

\begin{corollary}
The $4$-tuple $(\mathcal{B}_{n,\ell},I,supp,\mathcal{E})$
is a simple expansion set, where $\mathcal{B}_{n,\ell}$, $supp$, and $\mathcal{E}$ are as in Definitions \ref{definition:B} and \ref{definition:E}, and $I$ is the interval $[r_{1},r_{k}]$ (or $[r_{1},r_{k})$). 
\qed
\end{corollary}

\subsection{The CAT(0) cubical complex $\Delta(n,\ell)$}
\label{subsection:CAT(0)} 

\begin{definition} \cite{Far2} \label{definition:deltaB} (The simplicial complexes $\Delta_{\mathcal{B}}^{f}$ and $\Delta(n,\ell)$) Let $\mathcal{B}$ be an arbitrary simple expansion set.
A \emph{vertex} of $\Delta^{f}_{\mathcal{B}}$ is a finite subset
$\{ b_{1}, \ldots, b_{\alpha} \} \subseteq \mathcal{B}$ such that 
$\{ supp(b_{1}), \ldots, supp(b_{\alpha}) \}$ partitions 
the set $X$. 

For a given vertex $v_{1} = \{ b_{1}, \ldots, b_{\alpha} \}$ and a subset $v_{2} = 
\{ b_{i_{1}}, \ldots, b_{i_{\beta}} \} \subseteq v_{1}$ such that 
$|\mathcal{E}(b_{i_{\gamma}})| = 2$ for all $\gamma \in \{ 1, \ldots, \beta \}$, we define
\[ C(v_{1},v_{2}) = \{ \cup_{i=1}^{\alpha} E_{a_{i}}(b_{i}) \mid
 (a_{1}, a_{2}, \ldots, a_{\alpha}) \in \{ 0, 1 \}^{\alpha}; a_{i} = 0 \text{ if } b_{i} \not \in v_{2} \}. \]
Such a set $C(v_{1},v_{2})$ is a \emph{cube}. The members of 
$C(v_{1},v_{2})$ are in bijective correspondence with the $\alpha$-tuples
$(a_{1}, \ldots, a_{\alpha}) \in \{0, 1\}^{\alpha}$ such that $a_{i} = 0$ when $b_{i} \not \in v_{2}$. Thus, restricting to the coordinates $i$ such that $b_{i} \in v_{2}$, we obtain a natural bijection 
\[ C(v_{1},v_{2}) \rightarrow \{ 0, 1 \}^{\beta}. \]
We write $(x_{1}, \ldots, x_{\beta}) \leq (y_{1}, \ldots, y_{\beta})$ if
$x_{i} \leq y_{i}$ for $i = 1, \ldots, \beta$. This is clearly a partial order on $\{ 0, 1\}^{\beta}$. We give $[0,1]^{\beta}$ the simplicial complex structure in which vertices are members of $\{0,1\}^{\beta}$ and the simplices are ascending chains. 
The above bijection induces a simplicial complex structure on $C(v_{1},v_{2})$.

The simplices of $\Delta_{\mathcal{B}}^{f}$ are the simplices of $C(v_{1},v_{2})$,
as $v_{1}$ ranges over all vertices of $\Delta_{\mathcal{B}}^{f}$ and $v_{2}$ run over all subsets of $v_{1}$. We let 
$\mathcal{V}_{\mathcal{B}}$ and $\mathcal{S}_{\mathcal{B}}$ denote the vertices and simplices (respectively) of $\Delta_{\mathcal{B}}^{f}$. 

In the special case when $\mathcal{B} = \mathcal{B}_{n,\ell}$ we will write $\Delta(n,\ell)$ in place of $\Delta_{\mathcal{B}_{n,\ell}}^{f}$. 
\end{definition}

\begin{example} (An example of a cube in $\Delta(2,(0,1))$)
It may be useful at this point to consider a cube in a simple example. 
Assume that $n=2$ and $\ell = (0,1)$. In this case, the group 
$G_{n,\ell}$ is Thompson's group $F$. The labels in the tree $T_{n,\ell}$ are as follows. The root $\varepsilon$ is assigned no label. Each node with an address of the form $0^{m}$ ($m>0$) is reduced and is assigned the label $A$. Each node with an address of the form $1^{m}$ ($m>0$) is reduced and is assigned the label $B$. All other nodes have the label $x$. 

We let $v_{1} = \{ [id,I_{00}], [id,I_{01}], [id,I_{10}], [id,I_{11}] \}$ and $v_{2} = \{ [id,I_{01}], [id,I_{11}] \}$. In the current context, we can safely denote each $[id,I_{\omega}]$ by simply $\omega$. Thus,
$v_{1} = \{ 00, 01, 10, 11 \}$ and $v_{2} = \{ 01, 11 \}$. Note that 
\begin{align*}
E_{0}(01) &= \{ 01 \}; \\
E_{1}(01) &= \{ 010, 011 \}; \\
E_{0}(11) &= \{ 11 \}; \\
E_{1}(11) &= \{ 110, 111 \}.
\end{align*}  
The cube $C(v_{1},v_{2}) \subseteq \Delta(2,(0,1))$ is two-dimensional, and its corners are as follows:
\begin{itemize}
\item $\{ 00, 01, 10, 11 \}$;
\item $\{ 00, 010, 011, 10, 11 \}$;
\item $\{ 00, 01, 10, 110, 111 \}$;
\item $\{ 00, 010, 011, 10, 110, 111 \}$.
\end{itemize}
These corners correspond to the corners $(0,0)$, $(1,0)$, $(0,1)$, and $(1,1)$ (respectively) of the standard unit square $[0,1]^{2}$.
As a simplicial complex, $C(v_{1},v_{2})$ consists of two $2$-simplices meeting in the diagonal of $[0,1]^{2}$ from $(0,0)$ to $(1,1)$.
\end{example}

\begin{remark} (The complexes $\Delta_{\mathcal{B}}^{f}$ and
$\Delta_{\mathcal{B}}$)
The complex $\Delta_{\mathcal{B}}^{f}$ is called the \emph{full support complex} in \cite{Far} and \cite{Far2}, since $supp(v) = X$ for each vertex $v \in \Delta_{\mathcal{B}}^{f}$. A related complex, denoted by $\Delta_{\mathcal{B}}$, which allows vertices with arbitrary supports, is useful to define for technical reasons that do not directly concern us here. Notice that the definition 
of vertex in Definition \ref{definition:simpleexpansionsets} does not require $supp(v) = X$, and thus defines a vertex in $\Delta_{\mathcal{B}}$, but not necessarily one in $\Delta^{f}_{\mathcal{B}}$. 
\end{remark}

\begin{definition} \cite{Far2}
(The partial order $\preceq$)
Let $v', v'' \in \mathcal{V}_{\mathcal{B}}$. 
We write $v' \dot{\preceq} v''$ if $v'' \in C(v',v_{1})$, for some $v_{1} \subseteq v'$. Let $\preceq$ denote the transitive closure of $\dot{\preceq}$. The relation $\preceq$ is a partial order on the vertices of $\mathcal{V}_{\mathcal{B}}$.

We can describe $\preceq$ geometrically as follows. If $v', v'' \in \mathcal{V}_{\mathcal{B}}$, then $v' \preceq v''$ if and only if there is a sequence
\[ v' = v_{0}, v_{1}, v_{2}, v_{3}, \ldots, v_{m} = v'' \]
for $m \geq 0$ such that: (i) $v_{i}$ is connected to $v_{i+1}$ by an edge in $\Delta_{\mathcal{B}}^{f}$, and (ii) $|v_{i+1}| > |v_{i}|$, for $i=0, \ldots, m-1$. Thus, regarding the cardinality as a height function, we have $v' \preceq v''$ if and only if there is an ascending path from $v'$ to $v''$ in $\Delta_{\mathcal{B}}^{f}$.
\end{definition}

\begin{theorem} \cite{Far2} \label{theorem:CAT(0)}
(CAT(0) cubical complexes) Let $\mathcal{B}$ be an arbitrary simple expansion set. If the vertices of $\Delta_{\mathcal{B}}^{f}$ are 
a directed set with respect to $\preceq$, then $\Delta^{f}_{\mathcal{B}}$ is a CAT(0) cubical complex with respect to the cubes from Definition \ref{definition:deltaB}. \qed 
\end{theorem} 

\begin{theorem} (A CAT(0) construction for the groups $G_{n,\ell}$)
\label{theorem:CATSn,l} 
The complex $\Delta(n,\ell)$ is a CAT(0) cubical complex, where the cubes are as in 
Definition \ref{definition:deltaB}.
\end{theorem}

\begin{proof}
We first claim that, if $v \in \mathcal{V}_{\mathcal{B}}$, then $v \preceq v'$ for some $v'$
of the form 
\[ \{ [id_{D_{1}},D_{1}], \ldots, [id_{D_{\alpha}},D_{\alpha}] \}. \]
Assume that $v = \{ [f_{1},D_{1}], \ldots, [f_{d},D_{d}] \} \in \mathcal{V}_{\mathcal{B}}$.
Consider $[f_{1},D_{1}]$. We can assume that
$Dom(f_{1}) = D_{1}$, so $f_{1}$ is the union of
$s_{1}, \ldots, s_{\delta} \in S_{n,\ell}$. It follows that
$D_{1}$ is partitioned by the domains $D'_{1}, \ldots, D'_{\delta}$ of the $s_{i}$; thus, the domains of
the $s_{i}$ represent a cut set in the tree $T_{n,\ell}$ below the node $D_{1}$ (Lemma \ref{lemma:domandpar}). We can therefore perform a sequence of expansions at $\{ [f_{1},D_{1}] \}$, eventually resulting in the collection
\[ \{ [f_{1}, D'_{1}], \ldots, [f_{1}, D'_{\delta}] \} = \{ [s_{1},D'_{1}], \ldots, [s_{\delta},D'_{\delta}] \}. \]
Note that $[s_{i},D'_{i}] = [id, s_{i}(D'_{i})]$ by the definition of the equivalence relation on $\mathcal{B}_{n,\ell}$. Since $s_{i}(D'_{i})$ is a domain,
we have
\[ \{ [f_{1},D'_{1}], \ldots, [f_{1},D'_{\delta}] \} = \{ [id,D''_{1}], \ldots, [id,D''_{\delta}] \} = v'. \]
The expansions from $\{ [f_{1},D_{1}] \}$ to $v'$ show that
\[ \{ [f_{1},D_{1}], \ldots, [f_{d},D_{d}] \} \preceq v' \cup \{ [f_{2}, D_{2}], \ldots, [f_{d},D_{d}] \}. \]
We can then repeat this procedure for the remaining $[f_{i},D_{i}]$, proving the claim.

Now suppose that $v_{1}, v_{2} \in \mathcal{V}_{\mathcal{B}}$. We can replace $v_{i}$ with $v'_{i}$ ($i=1,2$), where $v_{i} \preceq v'_{i}$ and the members of
$v'_{i}$ all take the form $[id,D]$, where $D \in \mathcal{D}_{n,\ell}$. Thus, the vertices $v'_{1}$ and $v'_{2}$ represent partitions of $[r_{1},r_{k}]$ (or $[r_{1},r_{k})$) by members of $\mathcal{D}_{n,\ell}$. These partitions have a common refinement; the corresponding vertex $v''$ is such that
$v_{i} \preceq v'_{i} \preceq v''$ for $i=1,2$, completing the proof.
\end{proof}

\subsection{Orbits and Stabilizers} \label{subsection:orbstab} 

\begin{definition} \label{definition:domaintype}
(Domain type; types of cubes) Let  
$D_{1}, D_{2} \in \mathcal{D}^{+}_{n,\ell}$. 
We write $D_{1} \approx D_{2}$ if there is $s \in S_{n,\ell}$
such that $Dom(s) = D_{1}$ and $Im(s) = D_{2}$. 
The relation $\approx$ is an equivalence relation on $\mathcal{D}_{n,\ell}$. The 
equivalence classes are called \emph{domain types} and two members of the same equivalence class are said to have the \emph{same domain type}.
We extend $\approx$ to pairs $[f,D] \in \mathcal{B}$, letting
$[f_{1},D_{1}] \approx [f_{2},D_{2}]$ exactly if $D_{1} \approx D_{2}$. 

Two cubes $C(v_{1},v_{2})$, $C(v'_{1},v'_{2})$ have the same \emph{type} if $|v_{1}| = |v'_{1}|$ and the left-right order-preserving bijection $\phi: v_{1} \rightarrow v'_{1}$ is such that
\begin{enumerate}
\item $\phi$ restricts to a bijection from $v_{2}$ to $v'_{2}$, and
\item $b \approx \phi(b)$, for each $b \in v_{1}$.
\end{enumerate}
\end{definition}

\begin{definition} \label{definition:typeseq}
(The function $\mathcal{L}$; the type sequence of a cube)
Each $D \in \mathcal{D}^{+}_{n,\ell}$ determines a unique labelled reduced node $v$ of $T_{n,\ell}$ by Lemma \ref{lemma:domandpar}. We define 
$\mathcal{L}: \mathcal{D}^{+}_{n,\ell} \rightarrow \Sigma_{n,\ell}$ by letting $\mathcal{L}(D)$ be the label of the node $v$. We similarly let $\mathcal{L}([f,D]) = \mathcal{L}(D)$, for each $[f,D] \in \mathcal{B}$. 

Let $C(v_{1},v_{2})$ be a cube and assume $v_{1} = \{ b_{1}, \ldots, b_{k} \}$, where the $b_{i}$ are listed in left-to-right order. The \emph{type sequence} of
$C(v_{1},v_{2})$ is the pair
\[ (\mathcal{L}(b_{1}) \mathcal{L}(b_{2}) \ldots \mathcal{L}(b_{k}), \chi(v_{2}) ), \]
where $\chi(v_{2}) = \{ j \in \mathbb{N} \mid b_{j} \in v_{2}\}$.
\end{definition}

\begin{remark}
We think of the pair $(\omega, \chi)$ from Definition \ref{definition:typeseq} as a sequence in the following way. Suppose $\omega = AxxBxxxC$ and $\chi = \{ 2, 5, 8 \}$. We use a notational device, such as parentheses, to indicate the members of $\chi$:
\[ (AxxBxxxC, \{ 2, 5, 8 \}) = A(x)xB(x)xx(C). \]
\end{remark}

\begin{proposition} \label{proposition:action} (The action of $G_{n,\ell}$ on $\Delta(n,\ell)$) The function $\cdot : G_{n,\ell} \times \mathcal{B} \rightarrow \mathcal{B}$ defined by $g \cdot [f,D] = [gf,D]$ is a group action, which extends to an action of $G_{n,\ell}$ on $\Delta(n,\ell)$. The latter action is free and cube-permuting. Two cubes are in the same orbit if and only if they have the same type sequence.
\end{proposition}

\begin{proof}
The proof that $\cdot$ is a group action of $G_{n,\ell}$ on $\Delta(n,\ell)$ is straightforward (or see \cite{FH} or \cite{Far}). It is also clear that the action is cube-permuting. 

Next we show that the action of $G_{n,\ell}$ on vertices is free. Let $v = \{ b_{1}, \ldots, b_{\alpha} \}$ be such that the supports of the $b_{i}$ are arranged from left to right. Assume $g \in G_{n,\ell}$ is such that $g \cdot v = v$. It follows that $g \cdot b_{i} = b_{i}$ for $i=1, \ldots, \alpha$, since $g$ is order-preserving. Let 
$b_{i} = [f_{i}, D_{i}]$ where $Dom(f_{i}) = D_{i}$. We therefore have $[gf_{i},D_{i}] = [f_{i},D_{i}]$, so there is $h \in S_{n,\ell}$ such that (i) the domain and image of $h$ are $D_{i}$, and (ii) $gf_{i} = hf_{i}$. Condition (i) implies that $h = id_{D_{i}}$ and thus $gf_{i} = f_{i}$. This implies, in particular, that $g$ is the identity on $Im(f_{i})$. Since this is true for all $i$, and $\{ f_{1}(D_{1}), \ldots, f_{\alpha}(D_{\alpha}) \}$ is a partition of $Dom(g)$, $g = 1$. Thus, the action of $G_{n,\ell}$ on vertices is free.

Assume $g$ stabilizes a cube $C(v_{1},v_{2}) \subseteq \Delta(n,\ell)$. Since $g$ is height-preserving (i.e., cardinality preserving ), it will necessarily fix the unique vertex $v_{1} \in C(v_{1},v_{2})$ of minimal height. It then follows that $g=1$ by the above reasoning. Thus, $G_{n,\ell}$ acts freely on $\Delta(n,\ell)$. 

Assume that $v_{1}$ and $v_{2}$ have the same type.
Let $v_{1} = \{ [f_{1}, D_{1}], \ldots, [f_{\alpha},D_{\alpha}] \}$ and
$v_{2} = \{ [g_{1}, E_{1}], \ldots, [g_{\alpha},E_{\alpha}] \}$. It follows that, for $i=1, \ldots, \alpha$, there is $s_{i} \in S_{n, \ell}$  such that $s_{i}: D_{i} \rightarrow E_{i}$ is a bijection. We let $g \in G_{n,\ell}$ be the homeomorphism of $[r_{1},r_{k}]$ that restricts to $g_{i}s_{i}f_{i}^{-1}$ on the set $f_{i}(D_{i})$. We evaluate $g$ on the vertex $v_{1}$:
\begin{align*}
g \cdot v_{1} &= \{ [gf_{1},D_{1}], \ldots, [gf_{\alpha},D_{\alpha}] \} \\ 
&= \{ [g_{1}s_{1},D_{1}], \ldots, [g_{\alpha}s_{\alpha},D_{\alpha}] \} \\
&= \{ [g_{1},s_{1}(D_{1})], \ldots, [g_{\alpha},s_{\alpha}(D_{\alpha})] \} \\
&= v_{2}.
\end{align*}
(The second-to-last equality is due the definition of $\sim$.)
It follows that $v_{1}$ and $v_{2}$ are in the same $G_{n,\ell}$-orbit.

If $C(v_{1},v_{2})$ and $C(v'_{1},v'_{2})$ have the same type, then the $g$ that sends $v_{1}$ and $v'_{1}$ will also send $v_{2}$ to $v'_{2}$, since $g$ preserves the left-to-right ordering. It follows easily that $C(v_{1},v_{2})$ and $C(v'_{1},v'_{2})$ are in the same $G_{n,\ell}$-orbit as well.

Conversely, the action of $G_{n,\ell}$ is easily seen to preserve the type of a vertex.
\end{proof}

\begin{corollary} \label{corollary:groupiso}
(An isomorphism of groups)
We have an isomorphism
\[ G_{n,\ell} \cong \pi_{1}\left(G_{n,\ell} \backslash \Delta^{(2)}(n,\ell) \right), \]
where $G_{n,\ell} \backslash \Delta^{(2)}(n,\ell)$ is the quotient of the 
$2$-skeleton of $\Delta(n,\ell)$ by the left action of $G_{n,\ell}$. 
\end{corollary}

\begin{proof}
We recall that CAT(0) spaces are contractible \cite{BH}. Since the action of $G_{n,\ell}$ on $\Delta(n,\ell)$ is by covering transformations, we have an isomorphism $G_{n,\ell} \cong \pi_{1}(G_{n,\ell} \backslash \Delta(n,\ell))$ due to covering space theory (\cite{Hatcher, Munkres}). The corollary follows, since, by cellular approximation \cite{Ross}, $\Delta^{(2)}(n,\ell)$ and $\Delta(n,\ell)$ have isomorphic fundamental groups. Note that the cell structure of $\Delta(n,\ell)$ is assumed to be the cubical one. 
\end{proof}

\section{Isomorphism of $D(\mathcal{P}_{n,\ell},\omega_{n,\ell})$
with $G_{n,\ell}$} \label{section:isomorphism} 

\subsection{The correspondence between $\mathcal{V}_{\mathcal{B}_{n,\ell}}$ and $[\omega_{n,\ell}]$}

By the choice of $\omega_{n,\ell}$ (Subsection \ref{subsection:basepoint}), there is a partition $D_{1}, \ldots, D_{m}$ of $[r_{1},r_{k}]$, ordered from left-to-right, such that
\[ \mathcal{L}(D_{1})\mathcal{L}(D_{2})\ldots \mathcal{L}(D_{m})
\equiv \omega_{n,\ell}. \]
We let $v_{\ast} = \{ [id,D_{1}], \ldots, [id,D_{m}] \}$; clearly, $v_{\ast} \in 
\Delta(n,\ell)$ and $\mathcal{L}(v_{\ast}) \equiv \omega_{n,\ell}$. 

%

\begin{lemma} \label{lemma:image} (Codomain of $\mathcal{L}: \mathcal{V}_{\mathcal{B}_{n,\ell}} \rightarrow \Sigma^{+}_{n,\ell}$)
The function $\mathcal{L}: \mathcal{V}_{\mathcal{B}_{n,\ell}} \rightarrow \Sigma^{+}_{n,\ell}$ sends
$\mathcal{V}_{\mathcal{B}_{n,\ell}}$ into $[\omega_{n,\ell}]$.
\end{lemma}

\begin{proof}
As noted above, $\mathcal{L}(v_{\ast}) \equiv \omega_{n,\ell}$. It therefore suffices, by connectivity of $\Delta(n,\ell)$,  to show that the condition $\mathcal{L}(v) \in [\omega_{n,\ell}]$ is closed under adjacency in $\Delta(n,\ell)$. 

Thus, we assume that $v'$ and $v''$ are adjacent, and $\mathcal{L}(v') \in [\omega_{n,\ell}]$. Let $v' = \{ [f_{1},D_{1}], \ldots, [f_{m},D_{m}] \}$
and $\mathcal{L}(v') \equiv a_{1}\ldots a_{m}$. That is, $\mathcal{L}(D_{i}) \equiv a_{i}$ for all $i$.

Assume first that $v''$ is obtained from $v'$ by expansion at
$[f_{i},D_{i}]$. This means that $v''$ is the result of replacing the pair $[f_{i},D_{i}]$ with the members
of $E_{1}([f_{i},D_{i}])$. It follows directly that 
$\mathcal{L}(v'') \equiv a_{1}\ldots a_{i-1} w_{a_{i}} a_{i+1} \ldots a_{m}$, where $w_{a_{i}}$ is the word that labels the ordered children of the node $D_{i}$. Since $a_{i} \rightarrow w_{a_{i}}$  is a relation in $\mathcal{P}_{n,\ell}$,  
$\mathcal{L}(v'') \in [\omega_{n,\ell}]$.  

Assume that $v''$ is obtained from $v'$ by collapsing at a subset
\[ [f_{d},D_{d}], \ldots, [f_{d+e},D_{d+e}],\]
where a \emph{collapse} is the inverse of an expansion. There is a pair $[g_{j},E_{j}] \in v''$ such that $[g_{j},E_{j}]$ expands to 
\[ [f_{d},D_{d}], \ldots, [f_{d+e},D_{d+e}]. \]  Let $\mathcal{L}(E_{j}) \equiv c_{j}$
and let $w_{c_{j}}$ be the (collective) label of $[f_{d},D_{d}], \ldots, 
[f_{d+e},D_{d+e}]$. If 
\[ v'' = \{ [g_{1}, E_{1}], \ldots, [g_{p}, E_{p}] \} \]
and $\mathcal{L}(v'') \equiv c_{1}\ldots c_{p}$, then
$\mathcal{L}(v') \equiv c_{1} \ldots c_{j-1} w_{c_{j}} c_{j+1} \ldots c_{p}$. Thus, 
$\mathcal{L}(v'')$ is the result of applying the relation
$w_{c_{j}} \rightarrow c_{j}$ to $\mathcal{L}(v')$. Thus,
$\mathcal{L}(v'') \in [\omega_{n,\ell}]$.
\end{proof}

\begin{lemma} \label{lemma:surj}  (Surjectivity of $\mathcal{L}$)
The map $\mathcal{L}: \mathcal{V}_{\mathcal{B}_{n,\ell}} \rightarrow [\omega_{n,\ell}]$ is surjective.
\end{lemma}

\begin{proof}
Since $\mathcal{L}(v_{\ast}) \equiv \omega_{n,\ell}$, it suffices to show that the property of being in the image of $\mathcal{L}$ is closed under the application of any relation from $\mathcal{R}_{n,\ell}$. 

Assume that $\omega_{1} \equiv a_{1} \ldots a_{m}$ and 
$v = \{ [f_{1},D_{1}], \ldots, [f_{m},D_{m}] \}$ is such that 
$\mathcal{L}(v) \equiv \omega_{1}$. We must show:
(i) that if $\omega_{2} \equiv a_{1} \ldots a_{j-1} w_{a_{j}} a_{j+1} \ldots a_{m}$, where $a_{j} \rightarrow w_{a_{j}}$ is a relation in $\mathcal{R}_{n,\ell}$, then $\omega_{2} \in \mathcal{L}(\mathcal{V}_{\mathcal{B}_{n,\ell}})$, and (ii) that if
$\omega_{2} \equiv a_{1}\ldots a_{\alpha} c a_{\beta} \ldots a_{m}$, where
$c \rightarrow a_{\alpha+1} \ldots a_{\beta-1}$ is a relation in $\mathcal{R}_{n,\ell}$ ($\beta - \alpha >2$), then $\omega_{2} \in \mathcal{L}(\mathcal{V}_{\mathcal{B}_{n,\ell}})$.

Case (i) is straightforward: the expansion 
of $v$ at $[f_{j},D_{j}]$ directly yields a vertex $v'$ such that $\mathcal{L}(v') = \omega_{2}$, since the ordered children of the node $D_{j}$ are labelled $w_{a_{j}}$ by the definition of $\mathcal{P}_{n,\ell}$.

Case (ii) is slightly more subtle. We are given a relation 
$c \rightarrow a_{\alpha+1} \ldots a_{\beta-1}$ and must produce
$v'$ such that $\mathcal{L}(v') \equiv  a_{1}\ldots a_{\alpha} c a_{\beta} \ldots a_{m}$. Since $c \rightarrow a_{\alpha+1} \ldots a_{\beta-1} \in \mathcal{R}_{n,\ell}$, we can find a reduced node $E \in T_{n,\ell}$ labelled by $c$ whose children 
\[ E_{\alpha+1}, \ldots, E_{\beta-1}\] 
are collectively labelled $a_{\alpha+1} \ldots a_{\beta-1}$. It follows that
there are transformations $s_{\alpha+1}, \ldots, s_{\beta-1} \in S_{n,\ell}$
that match $E_{j}$ to $D_{j}$, for $j \in \{ \alpha+1, \ldots, \beta-1 \}$. 
It follows directly that
\[ [f_{j},D_{j}] = [f_{j},s_{j}(E_{j})] = [f_{j}s_{j},E_{j}], \]
for $j \in \{ \alpha+1, \ldots, \beta-1 \}$. We can now let 
$f$ be the disjoint union of the transformations 
$f_{j}s_{j}$ for $j \in \{ \alpha+1, \ldots, \beta-1 \}$. Define
\[ v' = \{ [f_{1},D_{1}], \ldots [f_{\alpha},D_{\alpha}], [f,E], [f_{\beta},D_{\beta}], \ldots, [f_{m},D_{m}] \}. \]
It now follows that $\mathcal{L}(v') \equiv \omega_{2}$. 
\end{proof}

\begin{proposition} \label{proposition:bij} 
(A bijection between vertices in $G_{n,\ell} \backslash \Delta(n,\ell)$ and $[\omega_{n,\ell}]$)
The function $\mathcal{L}: \mathcal{V}_{\mathcal{B}_{n,\ell}} \rightarrow [\omega_{n,\ell}]$
induces a bijection between the $G_{n,\ell}$-orbits in $\mathcal{V}_{\mathcal{B}_{n,\ell}}$ 
and $[\omega_{n,\ell}]$. 
\end{proposition}

\begin{proof}
This follows easily from Lemma \ref{lemma:surj}
and from the fact that two vertices $v', v''$ in $\mathcal{V}_{\mathcal{B}_{n,\ell}}$ have the same label if and only if they are in the same orbit under the action of 
$G_{n,\ell}$ (Proposition \ref{proposition:action}).
\end{proof}

\subsection{The map on the $2$-skeleton}

\begin{proposition} \label{proposition:graphiso}
(An isomorphism of graphs)
Let $C_{1}(\Delta(n,\ell))$ be the set of all $1$-cubes in 
$\Delta(n,\ell)$ and let $C_{1}(\mathcal{S}(\mathcal{P}_{n,\ell}, \omega_{n,\ell}))$ be the collection of $1$-cells in the connected component of the relevant Squier complex.  

Define $\mathcal{L}: C_{1}(\Delta(n,\ell)) \rightarrow
C_{1}(\mathcal{S}(\mathcal{P}_{n,\ell}, \omega_{n,\ell}))$ by sending a cube $C$ with type sequence \[ a_{1}\ldots a_{i-1} (a_{i}) a_{i+1} \ldots a_{m} \]
to the $1$-cell $(a_{1}\ldots a_{i-1}, a_{i} \rightarrow w_{a_{i}}, a_{i+1}\ldots a_{m})$.  

The function $\mathcal{L}$ induces an isomorphism between the graphs
$G_{n,\ell} \backslash (\Delta(n,\ell))^{(1)}$ 
and $\mathcal{S}(\mathcal{P}_{n,\ell},\omega_{n,\ell})^{(1)}$, extending the bijection from
Proposition \ref{proposition:bij}. 
\end{proposition}

\begin{proof}
We first argue that $\mathcal{L}$ determines a bijection between
$G_{n,\ell}$-orbits of $1$-cubes in $\Delta(n,\ell)$ and the $1$-cells of $\mathcal{S}(\mathcal{P}_{n,\ell},\omega_{n,\ell})$.

Let $e$ be an edge in $\mathcal{S}(\mathcal{P}_{n,\ell},\omega_{n,\ell})$. 
It follows that \[ e = (a_{1}\ldots a_{i-1}, a_{i} \rightarrow w_{a_{i}}, a_{i+1} \ldots a_{m}), \] for some $a_{1}, \ldots, a_{m} \in \Sigma_{n,\ell}$ and
some $i$, where $a_{1} \ldots a_{m} \in [\omega_{n,\ell}]$. By Proposition \ref{proposition:bij}, there is a vertex 
$v \in \Delta(n,\ell)$ such that 
$\mathcal{L}(v) \equiv a_{1} \ldots a_{m}$. Thus, $v = \{ b_{1}, \ldots, b_{m} \} \in \mathcal{V}_{\mathcal{B}_{n,\ell}}$, for some $b_{1}, \ldots, b_{m} \in \mathcal{B}_{n,\ell}$, where $\mathcal{L}(b_{i}) \equiv a_{i}$ for $i= 1, \ldots, m$. 
It follows directly that $\mathcal{L}$ carries the $1$-cube 
$C(v,\{ b_{i} \})$ to $e$. Thus, $\mathcal{L}$ is surjective.

Next, we recall that two cubes have the same type sequence if and only 
if they are in the same $G_{n,\ell}$-orbit (Proposition \ref{proposition:action}). It follows easily that
$\mathcal{L}$ determines a bijection between the $1$-cubes
in $\Delta(n,\ell)$ and the $1$-cells in $\mathcal{S}(\mathcal{P}_{n, \ell},\omega_{n,\ell})$. 

It remains to show that this bijection 
preserves the required incidence relation between $1$-cubes and vertices. This is straightforward. If $e=(a_{1}\ldots a_{i-1}, a_{i} \rightarrow w_{a_{i}}, a_{i+1}\ldots a_{m})$ is a $1$-cell
in $\mathcal{S}(\mathcal{P}_{n,\ell},\omega_{n,\ell})$, then its preimage in $G_{n,\ell} \backslash \Delta(n,\ell)$ is represented by a cube 
$C(v, \{ b_{i} \})$, where $v = \{ b_{1}, \ldots, b_{m} \}$ and $\mathcal{L}(b_{j}) \equiv a_{j}$ for all $j$. The initial vertex of $e$ is $a_{1}\ldots a_{m}$ and the terminal vertex
of $e$ is $a_{1}\ldots a_{i-1} w_{a_{i}} a_{i+1} \ldots a_{m}$. The labels of the initial and terminal vertices of $C(v, \{ b_{i} \})$
are likewise $a_{1} \ldots a_{m}$ and $a_{1} \ldots a_{i-1} w_{a_{i}} a_{i+1} \ldots a_{m}$ (respectively). This completes the proof.  
\end{proof}

\begin{figure}[!b] 
\begin{center} 
\begin{tikzpicture}

\filldraw[lightgray,thick] (0,0) -- (2,2) -- (0,4) -- (-2,2) --  cycle; 
\draw[black, very thick] (0,0) -- (2,2) -- (0,4) -- (-2,2) -- cycle; 
\draw[black, very thick, ->] (0,0) -- (1.1,1.1); 
\node at (1.8,.7){$a a_{i} b (a_{j}) c$}; 
\draw[black, very thick, ->] (-2,2) -- (-.9,3.1); 
\node at (-1.8,3.35) {$a w_{a_{i}} b (a_{j}) c$};
\draw[black, very thick, ->] (0,0) -- (-1.1,1.1); 
\node at (-1.8,.7){$a(a_{i})b a_{j} c$};
\draw[black, very thick, ->] (2,2)-- (.9,3.1); 
\node at (1.8,3.35){$a (a_{i}) b w_{a_{j}} c$   };

\node at (0,-.4){$a a_{i} b a_{j} c$}; 
\node at (2.9,2){$a a_{i} b w_{a_{j}} c$}; 
\node at (-2.9,2){$a w_{a_{i}} b a_{j} c$}; 
\node at (0,4.3){$a w_{a_{i}} b w_{a_{j}} c$}; 
\end{tikzpicture}
\end{center}
\caption{The attaching map of a quotient $2$-cube in the
complex $G_{n,\ell} \backslash \Delta^{(2)}(n,\ell)$.}
\label{quotient}
\end{figure}
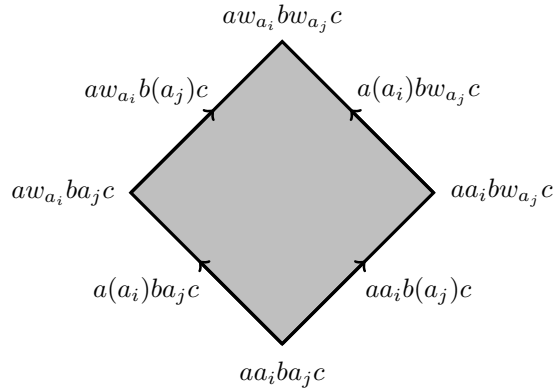

\begin{proposition} (Isomorphism between the $2$-complexes) \label{proposition:compiso}
The map from Proposition \ref{proposition:graphiso} extends to an isomorphism between $G_{n,\ell}\backslash \Delta^{(2)}(n,\ell)$ 
and $\mathcal{S}(\mathcal{P}_{n,\ell},\omega_{n,\ell})$.
\end{proposition}

\begin{proof}
Let $C(v_{1},v_{2})$ be a $2$-cube in $\Delta(n,\ell)$. Assume that $v_{1} = \{ b_{1}, \ldots, b_{m} \}$ and $v_{2} = \{ b_{i}, b_{j} \}$, where $i < j$. For each 
$\alpha \in \{ 1, \ldots, m \}$, we suppose $\mathcal{L}(b_{\alpha}) \equiv a_{\alpha}$, where $a_{\alpha} \in \Sigma_{n,\ell}$. We can now define a
map $\mathcal{L}: C_{2}(\Delta(n,\ell)) \rightarrow
C_{2}(\mathcal{S}(\mathcal{P}_{n,\ell}, \omega_{n,\ell}))$ between the sets of $2$-cells by the rule
\[ \mathcal{L}(C(v_{1},v_{2})) = (a_{1}\ldots a_{i-1}, a_{i} \rightarrow w_{a_{i}}, a_{i+1}\ldots a_{j-1}, a_{j} \rightarrow w_{a_{j}}, a_{j+1}\ldots a_{m}). \]
This map descends to a well-defined injection
$\mathcal{L}: C_{2}(G_{n,\ell} \backslash \Delta(n,\ell))
\rightarrow \mathcal{S}(\mathcal{P}_{n,\ell},\omega_{n,\ell})$, since 
two cubes are in the same $G_{n,\ell}$-orbit exactly when 
they have the same type sequence (Proposition \ref{proposition:action}). The map $\mathcal{L}$ is also surjective on $2$-cells, for essentially the same reason that $\mathcal{L}$ was surjective on $1$-cells. Thus, $\mathcal{L}$ maps $2$-cubes
of $G_{n,\ell} \backslash \Delta(n,\ell)$ to $2$-cubes
of $\mathcal{S}(\mathcal{P}_{n,\ell}, \omega_{n,\ell})$ bijectively.

Finally, we must check the compatibility of the attaching maps. We write
$a = a_{1}\ldots a_{i-1}$, $b = a_{i+1} \ldots a_{j-1}$, $c = a_{j+1} \ldots a_{m}$. The image of $C(v_{1}, v_{2})$ in $G_{n,\ell} \backslash \Delta(n,\ell)$ takes the form depicted in Figure \ref{quotient}. Comparing this singular $2$-cube with its 
image $(a,a_{i} \rightarrow w_{a_{i}}, b, a_{j} \rightarrow w_{a_{j}}, c)$ in $\mathcal{S}(\mathcal{P}, \omega_{n,\ell})$ (see Figure \ref{squier}), we find that the squares are essentially identical up to relabelling. \end{proof}

\begin{corollary}
(Isomorphism between $G_{n,\ell}$ and $D(\mathcal{P}_{n,\ell}, \omega_{n,\ell})$)
We have an isomorphism of groups:
\[ G_{n,\ell} \cong D(\mathcal{P}_{n,\ell}, \omega_{n,\ell}). \]
\end{corollary}

\begin{proof}
This follows by combining Corollary \ref{corollary:groupiso}, Proposition \ref{proposition:compiso}, and Definition \ref{definition:Dpw}.
\end{proof}

\section{Finiteness properties} \label{section:finf}

The paper \cite{Far} created a framework for establishing finiteness properties of the groups $\Gamma(S)$ that act on 
the complexes $\Delta^{f}_{\mathcal{B}}$. This framework uses the well-known finiteness criterion due to Brown \cite{B}
in an essential way. We summarise as much of this material as we will need in Subsection \ref{subsection:template}, and then establish finiteness properties of the groups $G_{n,\ell}$ in Subsection \ref{subsection:finf}.

\subsection{A template for proving $F_{\infty}$}
\label{subsection:template}
\begin{convention} \label{convention:SShatGamma} 
(The inverse semigroups $S$ and $\widehat{S}$) 
Assume  that an expansion set $\mathcal{B}$ over a linearly ordered set $X$ and an inverse semigroup $S$ 
acting on $X$ are given.

We assume that $\mathcal{B}$ is defined in the ``usual" way:
i.e., members of $\mathcal{B}$ are equivalence classes 
$[f,D]$, where $f \in \widehat{S}$ and $D \in \mathcal{D}^{+}(S)$. Two such classes $[f_{1},D_{1}]$ and $[f_{2},D_{2}]$ are the same if there is $s: D_{1} \rightarrow D_{2}$ in $S$ such that $f_{2}s = f_{1}$ on $D_{1}$. See Definition \ref{definition:hatSnl}.

We construct the inverse semigroup $\widehat{S}$ to have the following properties:
\begin{enumerate}
\item $S \subseteq \widehat{S}$;
\item each $\hat{s} \in \widehat{S}$ is a finite disjoint union of members of $S$;
\item each $\hat{s} \in \widehat{S}$ is continuous and order-preserving;
\item $\widehat{S}$ is closed under (continuous) disjoint unions: i.e., if $\hat{s}_{1}, \hat{s}_{2} \in \widehat{S}$ have disjoint domains and images, then the disjoint union 
$\hat{s}_{1} \coprod \hat{s}_{2}$ is in $\widehat{S}$ provided that the function in question is continuous and has an interval as its domain;
\item $\widehat{S}$ has a non-empty subset of full-support elements; i.e., bijections from $X$ to $X$. Such elements obviously form a group, which we have denoted by 
$\Gamma(S)$ (Definition \ref{definition:Ginv}) ;
\item for each $b \in \mathcal{B}$, $supp(b)$ is the domain of some $\hat{s} \in \widehat{S}$. 
\end{enumerate}
The inverse semigroups $\widehat{S}_{n,\ell}$ (Definition \ref{definition:hatSnl}) all satisfy the above conditions on $\widehat{S}$.

We assume that the above conventions are in place throughout the rest of this subsection.
\end{convention}

\begin{definition} \label{definition:collection} (A collection of definitions)
\begin{enumerate}
\item ($lk_{\uparrow}(b,v)$) if $b \in \mathcal{B}$ and $v$ is a vertex such that $b \precneq v$, then 
$lk_{\uparrow}(b,v)$ is the simplicial complex whose vertex set consists of vertices $v'$ in $\mathcal{E}(b)$ such that $b \precneq v' \preceq v$, and whose simplices are $\preceq$-ascending 
chains in $\mathcal{E}(b)$ on those vertices;
\item ($\widehat{S}_{b}$) if $b = [f,D] \in \mathcal{B}$, then 
$\widehat{S}_{b}$ is the set of all $\hat{s} \in \widehat{S}$
such that: (i) the domain and image of $\hat{s}$ are both 
$f(D)$, and (ii) $[\hat{s}f,D] = [f,D]$;
\item (consecutively ordered) A vertex $v = \{ b_{1}, \ldots, b_{m} \}$ is \emph{consecutively ordered} if $supp(v)$ is an interval, and the supports $supp(b_{1})$, $supp(b_{2})$, $supp(b_{3})$, $\ldots$ are arranged from left to right.
(We note that $supp(v)$ need not be all of $[r_{1},r_{k}]$.)
\item (rich in contractions) We say that $\mathcal{B}$ is \emph{rich in contractions} if there is a constant $C_{1}$ such that, for any consecutively ordered vertex $v = \{ b_{1}, \ldots, b_{m} \}$ $(m \geq C_{1})$, there is always a consecutively ordered subset $v' = \{ b_{\alpha}, \ldots, b_{\beta} \}$ such that $v' \in \mathcal{E}(b') - \{ b' \}$, for some $b' \in \mathcal{B}$;
\item (bounded contractions) We say that $\mathcal{B}$  has
the \emph{bounded contractions property} if there is a constant $C_{0}$ such that, for all $b \in \mathcal{B}$ and for all $v \in \mathcal{E}(b)$, $|v| \leq C_{0}$.
\end{enumerate}
\end{definition}

\begin{theorem} \cite{Far} \label{theorem:template}  (A template for proving that $\Gamma$ has type $F_{\infty}$)
Let $\mathcal{B}$ be an expansion set over $X$; let $\widehat{S}$ be an inverse semigroup acting on $\mathcal{B}$. Let $\Gamma(S)$ be the full-support subgroup of $\widehat{S}$. If
\begin{enumerate}
\item the vertices of $\Delta^{f}_{\mathcal{B}}$ are a directed set with respect to $\preceq$;
\item for each $b \in \mathcal{B}$ and $v \in \mathcal{V}_{\mathcal{B}}$ such that $\{ b \} \precneq v$, 
$lk_{\uparrow}(b,v)$ is contractible; 
\item each stabilizer group $\widehat{S}_{b}$ ($b \in \mathcal{B}$) has type $F_{\infty}$, and acts cocompactly on $\mathcal{E}(b)$;
\item the action of $\Gamma$ on $\mathcal{B}$ has finitely many orbits;
\item $\mathcal{B}$ is rich in contractions and has the bounded contractions property,
\end{enumerate}
then $\Gamma$ has type $F_{\infty}$.
\end{theorem}

\subsection{Finiteness properties of $G_{n,\ell}$}
\label{subsection:finf}

\begin{theorem} \label{theorem:finf}(Finiteness properties of $G_{n,\ell}$)
If $\ell$ is a sequence of rational numbers, then $G_{n,\ell}$ has type $F_{\infty}$. If at least one entry in $\ell$ is irrational, then $G_{n,\ell}$ is not
finitely generated.
\end{theorem}

\begin{proof}
Let $\ell = (r_{1}, \ldots, r_{k})$. If $r_{i}$ is irrational for some $i \in \{ 1, \ldots, k \}$, then every $h \in G_{n,\ell}$ fixes an open neighborhood of $r_{i}$ by Lemma \ref{lemma:basic}(3-4). Thus, if $G_{n,\ell}$ were finitely generated, then there would exist an open neighborhood
of $r_{i}$ on which every $h \in G_{n,\ell}$ is the identity. It is, however, straightforward to argue that there are homeomorphisms $g \in G_{n,\ell}$ that move points arbitrarily close to $r_{i}$. This is a contradiction, and proves that $G_{n,\ell}$ is infinitely generated. 

Now assume that each $r_{i}$ is rational. We consider points (1)-(5) from 
Theorem \ref{theorem:template}. Point (1) is covered in the proof of Theorem
\ref{theorem:CATSn,l}. If $\{ b \} \precneq v$, then there is a sequence 
\[ \{ b \} = v_{0}, v_{1}, \ldots, v_{m} =v \]
such that, for all $i$, $v_{i+1} \in C(v_{i}, v'_{i})$. We note that this condition completely determines $v_{1}$; indeed, we must always have $v_{1} = E_{1}(b)$, for any sequence as above. It now follows from Definition \ref{definition:collection}(1) that $lk(b,v)$ is a point, and therefore contractible, proving (2).   The stabilizer group $\widehat{S}_{b}$ is, by definition, the collection of all transformations $\widehat{s}$ in $\widehat{S}_{n,\ell}$
such that: (i) the domain and range of $\widehat{s}$ is $supp(b)$, and 
(ii) $\widehat{s} \cdot b = b$. Let $[f,D] \in \mathcal{B}$, and let $\widehat{s} \in \widehat{S}_{b}$. It follows that $[f,D] = [\widehat{s}f, D]$, so there is some $h: D \rightarrow D$ $(h \in S_{n,\ell})$ such that $\widehat{s}fh_{\mid D} = f$. We note that $h$ is necessarily the identity (Lemma \ref{lemma:basic}(4)), so we have the equality 
$\widehat{s}f = f$. If $y \in f(D)$, then there is $x \in D$ such that
$f(x) = y$. Evaluating $\widehat{s}$ at $y$, we find $\widehat{s}(y) = y$, so
$\widehat{s}$ is the identity on all of $f(D)$. Thus, $\widehat{S}_{b}$ is the trivial group. Each $\mathcal{E}(b)$ consists of two points, so (3) is satisfied. Point (4) is satisfied since the orbits
of $\mathcal{B}_{n,\ell}$ under the action of $G_{n,\ell}$ are in one-to-one correspondence with the letters of $\Sigma_{n,\ell}$, which are finite in number under the current hypotheses.

Finally, we consider (5). For a given $b \in \mathcal{B}_{n,\ell}$, 
\[ \mathcal{E}(b) = \{ E_{0}(b), E_{1}(b) \}, \]
where $E_{0}(b)$ contains one member, and $E_{1}(b)$ contains no more than $n$ members. Thus, the bounded contractions property is satisfied with $C_{0} =n$. 
We claim that $\mathcal{B}$ is rich in contractions with constant $C_{1} = kn-n+2$. Indeed, suppose that the vertex $v = \{ b_{1}, \ldots, b_{\beta} \}$ is consecutively ordered and $\beta \geq kn-n+2$. It suffices to argue that
\[ \mathcal{L}(v) = \mathcal{L}(b_{1}) \ldots \mathcal{L}(b_{\beta}) \]
contains a subword of the form $x^{n}$, for then we are guaranteed the required contraction by case (ii) in the proof of Lemma \ref{lemma:surj}. We consider the maximal subwords of $\mathcal{L}(v)$ having the form $x^{\alpha}$, for some $\alpha > 0$. There are at most $k-1$ such words, since any two maximal subwords of the 
form $x^{\alpha'}$ and $x^{\alpha''}$ must be separated by a label from 
one of the automata $A_{n}(r_{m})$, with $m \in \{ 2, \ldots, k-1 \}$, and the label $\mathcal{L}(v)$ cannot contain more than one label from a given automaton $A_{n}(r_{m})$. Thus, if there were no subword of the form $x^{n}$, we would have at most 
$(k-1)(n-1)$ occurrences of ``$x$'' in the word $\mathcal{L}(v)$. However, the remaining letters of $\mathcal{L}(v)$ can be no more than $k$ in number, 
since each letter is chosen from the labels of a given automaton
$A_{n}(r_{m})$, and at most one label can be selected from each automaton.
It follows that $\mathcal{L}(v)$ contains no more than $(k-1)(n-1)+k = kn-n+1$ symbols. However, $\mathcal{L}(v) = \beta > kn-n+1$ by hypothesis. This is a contradiction. This proves (5).

Since (1)-(5) are satisfied, it now follows from Theorem \ref{theorem:template} that $G_{n,\ell}$ has type $F_{\infty}$ when all of the $r_{i}$ are rational.

\end{proof}

\section{Alternate approaches and further directions}
\label{section:alternate}

In this final section, we consider different approaches to the proof of our main theorem and some additional directions for research.

\subsection{Non-complete semigroup presentations}

Most of the semigroup presentations $\mathcal{P}_{n,\ell}$ are not complete. For instance, consider the semigroup presentation $\mathcal{P}_{3,\ell}$ from Section \ref{section:example}. First, reverse each arrow in $\mathcal{R}$ (so that $x \rightarrow x^{3}$ becomes $x^{3} \rightarrow x$) in order to create a terminating presentation. It is straightforward to check that the word $AExG$ is equivalent modulo $\mathcal{P}_{3,\ell}$ to the reduced words $AEG$ and $BDG$. It follows that $\mathcal{P}_{3,\ell}$ cannot be complete. There are many more examples of this kind.

Guba and Sapir found methods to compute group presentations \cite{GS} and homology groups \cite{GS2}
of diagram groups $D(\mathcal{P},\omega)$. Their methods, however, apply most directly to complete semigroup presentations. The following problem may therefore  be considered open.
\begin{problem}\label{problem:myp}
Compute the homology groups $H_{\ast}(G_{n, \ell}; \mathbb{Z})$ for general $n$ and $\ell$. Compute group presentations for the groups $G_{n,\ell}$.
\end{problem}
Guba and Sapir \cite{GS2} use Brown's idea of a collapsing scheme \cite{BrownCollapse} to facilitate the computation of the homology of $D(\mathcal{P},\omega)$. The hypothesis of completeness ensures that the homology groups can be computed from a chain complex in which all boundary maps are $0$. The same methods appear to be applicable in the current context, but one will probably need to deal with non-zero boundary maps in most cases.

Note that \cite{GolanSapir2} introduces methods to compute presentations for $D(\mathcal{P},\omega)$ under the weaker hypothesis that $\mathcal{P}$ is a semi-complete presentation. The latter may offer a useful procedure for computing presentations of the groups $G_{n,\ell}$. 

\subsection{Closed subgroups of $F$ and $G_{n,\ell}$}

In \cite{GolanSapir2}, Golan-Polak
and Sapir define closed subgroups of Thompson's group $F$. These are subgroups of $F$ that have a natural description as diagram groups (see Lemma 2.6 and Definition 2.7 from \cite{GolanSapir2}). After Lemma 6.2, they note that if $S$ is
a subset of $[0,1]$, then the stabilizer group $F_{S}$ is a closed subgroup (i.e., a diagram group). Thus, the 
statement that $G_{n,\ell}$ is a diagram group
is originally due to Golan-Polak and Sapir.

\subsection{A different proof of the $F_{\infty}$ property}

After an earlier version of this paper was posted to the arXiv, I learned from Matt Brin the following fact and an alternate proof of the  $F_{\infty}$ portion of Theorem \ref{theorem:bigone}.

\begin{proposition} \label{proposition:Brin}
If $a<b$ are rational numbers, then 
\[ F \cong G_{2,(a,b)}. \]
\end{proposition}

This fact seems to be well-known to some people, but it is not clear whether a proof has been published. The following proof is adapted from notes due to  Brin, which were based on proofs supplied independently to him by Collin Bleak and Yash Lodha. 

\begin{proof}
We need to set up our argument with help from a few carefully chosen members of $F$. We let
$\ell = x_{0}x_{1}^{-1}$ and $r = x_{0}x_{1}x^{-1}_{0}$, where $x_{0}$ and $x_{1}$ are the standard generators of $F$ as described in \cite{CFP}. (Our convention, as always, is to write functions on the left.) We let $u = r \ell^{-1} r^{-1} \ell$ and $v = \ell^{-1} r \ell^{-1} r^{-1} \ell^{2}$. A straightforward check shows that:
\begin{enumerate}
\item $supp(\ell) = (0,3/4)$ and $supp(r) = (1/4,1)$;
\item $F = \langle \ell, r \rangle$;
\item $F_{[1/4,3/4]} = \langle u, v \rangle$.
\end{enumerate}
Here $supp(f)$, the \emph{support} of $f \in F$, is the set of all $x \in [0,1]$ such that $
f(x) \neq x$. The group $F_{[a,b]}$ has been denoted by $G_{2,(a,b)}$ in the main body of the paper. 

We now record, for future reference, the itineraries of the points $1/4$ and $3/4$ under the words $u$ and $v$. We list each intermediate destination of these points as the generators $\ell$ and $r$ are applied:
\begin{align*}
\\
u: \quad \quad & \frac{1}{4} \xrightarrow{\ell} \frac{1}{8} \xrightarrow{r^{-1}} \frac{1}{8} \xrightarrow{\ell^{-1}} \frac{1}{4} \xrightarrow{r} \frac{1}{4}; \quad \quad
\frac{3}{4} \xrightarrow{\ell} \frac{3}{4} \xrightarrow{r^{-1}} \frac{7}{8} \xrightarrow{\ell^{-1}} \frac{7}{8} \xrightarrow{r} \frac{3}{4}. \\ \\
v: \quad \quad & \frac{1}{4} \xrightarrow{\ell} \frac{1}{8} \xrightarrow{\ell} \frac{1}{16}
\xrightarrow{r^{-1}} \frac{1}{16} \xrightarrow{\ell^{-1}} \frac{1}{8} \xrightarrow{r} \frac{1}{8} \xrightarrow{\ell^{-1}} \frac{1}{4};\\  \\
&\frac{3}{4} \xrightarrow{\ell} \frac{3}{4} \xrightarrow{\ell} \frac{3}{4}
\xrightarrow{r^{-1}} \frac{7}{8} \xrightarrow{\ell^{-1}} \frac{7}{8} \xrightarrow{r} \frac{3}{4} \xrightarrow{\ell^{-1}} \frac{3}{4}. \\
\end{align*}
The strategy of our proof will involve altering the generators $\ell$ and $r$ into
a different pair $\hat{\ell}$ and $\hat{r}$ with desirable properties.
Suppose that $\hat{\ell}, \hat{r} \in F$ are such that
\begin{enumerate}
\item $supp(\hat{\ell}) \subseteq supp(\ell)$ and $supp(\hat{r}) \subseteq supp(r)$, and 
\item the restrictions of $\hat{\ell}$ and $\hat{r}$ to $[1/16, 7/8]$ agree with the restrictions of $\ell$ and $r$ to $[1/16,7/8]$. 
\end{enumerate}
We claim that $\langle \hat{\ell}, \hat{r} \rangle$ is isomorphic to $F$, by an isomorphism $\phi$ that sends $\ell$ to $\hat{\ell}$ and $r$ to $\hat{r}$, and that
$\langle \hat{u}, \hat{v} \rangle = F_{[1/4,3/4]}$, where $\hat{u} = \phi(u)$ and $\hat{v} = \phi(v)$.  

We first prove that $\phi$ is an isomorphism. The group $F$ has the following presentation relative to the generators $\ell$ and $r$:
\[ F \cong \langle \ell, r \mid [\ell, r^{\ell r \ell}], [\ell, r^{\ell r \ell r \ell}] \rangle. \]
(Here $a^{b} := b^{-1}ab$.) The latter presentation is the result 
of replacing each occurrence of $x_{0}$ with $r\ell$ and each occurrence of $x_{1}$ with $\ell^{-1} r \ell$ in a well-known presentation of $F$ \cite{CFP}. Thus, to show that $\phi$ is a homomorphism, it suffices to show that $\hat{\ell}$ and $\hat{r}$ satisfy the analogous relations. For this, it is enough to argue that
\begin{itemize}
\item $supp(\hat{\ell}) \cap \hat{\ell}^{-1} \hat{r}^{-1} \hat{\ell}^{-1} \cdot supp(\hat{r}) = \emptyset$, and
\item $supp(\hat{\ell}) \cap \hat{\ell}^{-1} \hat{r}^{-1} \hat{\ell}^{-1} \hat{r}^{-1} \hat{\ell}^{-1} \cdot supp(\hat{r}) = \emptyset$.
\end{itemize}
Indeed, it suffices to prove the above with $supp(\ell) = (0,3/4)$ and $supp(r)$ taking the place of $supp(\hat{\ell})$ and $supp(\hat{r})$, by assumption (1). We are thus led to consider the itinerary of $1/4$ under the words $\ell^{-1} r^{-1} \ell^{-1} r^{-1} \ell^{-1}$ and $\ell^{-1} r^{-1} \ell^{-1}$:
\[ \frac{1}{4} \xrightarrow{\ell^{-1}} \frac{1}{2} \xrightarrow{r^{-1}} \frac{3}{4} \xrightarrow{\ell^{-1}} \frac{3}{4} \xrightarrow{r^{-1}} \frac{7}{8} 
\xrightarrow{\ell^{-1}} \frac{7}{8}. \]
Since this itinerary lies entirely inside $[1/16,7/8]$, it follows that the itinerary 
with respect to the generators $\hat{\ell}$ and $\hat{r}$ is identical, by assumption (2). It now follows easily that 
\[ \hat{\ell}^{-1} \hat{r}^{-1} \hat{\ell}^{-1} \cdot supp(r) = (3/4,1), 
\quad \text{ and } \quad \hat{\ell}^{-1} \hat{r}^{-1} \hat{\ell}^{-1} \hat{r}^{-1} \hat{\ell}^{-1} \cdot supp(r) = (7/8,1), \]
from which the desired conclusion follows readily. Thus, $\phi$ is a homomorphism. Clearly it is surjective. If $\phi$ were not injective, then 
$\langle \hat{\ell}, \hat{r} \rangle$ would be a proper quotient of $F$, and therefore abelian \cite{CFP}. This is not the case, since (for instance)
$\hat{r} \hat{\ell} (1/2) = 1/4$, but $\hat{\ell} \hat{r} (1/2) = 3/16$. Thus,
$\phi$ is an isomorphism.

We now prove the equality $\langle \hat{u}, \hat{v} \rangle = F_{[1/4,3/4]}$.
We first note that the restrictions of $\hat{u}$ and $\hat{v}$ to $[1/4,3/4]$ must agree with those of $u$ and $v$, since the itinerary of any $x \in [1/4,3/4]$ relative to $u$ or $v$ always stays inside of the the interval $[1/16,7/8]$, and thus the corresponding itineraries relative to $\hat{u}$ or $\hat{v}$ are identical, by assumption (2). Now we must argue that $\hat{u}$ and $\hat{v}$ 
act as the identity outside of $[1/4,3/4]$. Here, we argue only a representative case. Suppose that $x < 1/4$ and consider the action of $\hat{u}$. We know that $\hat{\ell}(x) < \hat{\ell}(1/4) = 1/8$. We therefore have
$\hat{r}^{-1} \hat{\ell}(x) = \hat{\ell}(x)$, since $\hat{\ell}(x)$ lies outside of the support of $\hat{r}^{-1}$. Thus
\begin{align*}
\hat{\ell}^{-1} \hat{r}^{-1} \hat{\ell}(x) = x \quad &\Rightarrow \quad 
\hat{r}\hat{\ell}^{-1} \hat{r}^{-1} \hat{\ell}(x) = \hat{r}(x) \\
&\Rightarrow \quad \hat{r}\hat{\ell}^{-1} \hat{r}^{-1} \hat{\ell}(x) = x,
\end{align*}
where the equality $\hat{r}(x) = x$ is due to the fact that $x \not \in supp(\hat{r})$. Thus, $\hat{u}(x) = x$. This completes the proof of the claim.

We are now ready to prove Proposition \ref{proposition:Brin}.
Assume that $p<q$ are rational numbers in $(0,1)$. We choose
$p',q' \in \mathbb{Q}$ and $\hat{\ell}, \hat{r} \in F$ such that:
\begin{itemize}
\item $p'$ and $q'$ are in the same $F$-orbit as $p$ and $q$, respectively;
\item $\hat{\ell}$ agrees with $\ell$ except in a small neighborhood $U$ of 
$p'$, where $U \subseteq [0,1/16)$;
\item $\hat{r}$ agrees with $r$ except in a small neighborhood $V$ of $q'$,
where $V \subseteq (7/8,1]$;
\item $\hat{\ell}(p') = p'$ and $\hat{\ell}'(p')= 2^{\alpha}$, where $\alpha$ is the negative of the length of the loop in the automaton $A_{2}(p')$;
\item $\hat{r}(q') = q'$, $\hat{r} = 2^{\beta}$, where $\beta$ is the  length of the loop in the automaton $A_{2}(q')$;
\item $(p',3/4) \subseteq supp(\hat{\ell})$ and $(1/4,q') \subseteq supp(\hat{r})$.
\end{itemize}
Under these assumptions, $\langle \hat{\ell}, \hat{r} \rangle \cong F$ and 
$F_{[1/4,3/4]} \subseteq \langle \hat{\ell}, \hat{r} \rangle$ by the previous claim. Restricting the members of $\langle \hat{\ell}, \hat{r} \rangle$ to $[p',q']$, we get an injective homomorphism $\psi: \langle \hat{\ell}, \hat{r} \rangle \rightarrow F_{[p',q']}$. We will identify $\hat{\ell}$ and $\hat{r}$ with their images in $F_{[p',q']}$. 

We claim that $\psi$ is surjective. Let $\omega \in F_{[p',q']}$. By the choices of the derivatives $\hat{\ell}$ and $\hat{r}$ at $p'$ and $q'$, we can find suitable exponents $a$ and $b$ such that $\omega_{1} := \hat{\ell}^{a} \hat{r}^{b} \omega$
is the identity in small neighborhoods of $p'$ and $q'$. We can then conjugate
 $\omega_{1}$ so that its support lies inside $(1/4,3/4)$, by a word $\tau \in \langle \hat{\ell}, \hat{r} \rangle$. Thus, $\omega_{2} := \tau \omega_{1} \tau^{-1} \in F_{[1/4,3/4]} = \langle \hat{u}, \hat{v} \rangle$. It follows that $\omega \in \langle \hat{\ell}, \hat{r} \rangle$, proving surjectivity. Thus, $F \cong F_{[p',q']} \cong F_{[p,q]}$. 
 
 The case in which $p=0$ or $q=1$ is easier to handle. One simply lets $\hat{\ell} = \ell$ and/or $\hat{r} = r$. 
\end{proof}

Using Proposition \ref{proposition:Brin}, we can prove the $F_{\infty}$ property for the groups $G_{2,\ell}$ as follows.
Assume, for the sake of simplicity, that $\ell = (a,b,c)$,
where $a<b<c$ are all rational. Assume also that $b$ is not a dyadic rational number, since this is the more difficult case. Let $F_{[a,b]}$ and $F_{[b,c]}$ denote the groups $G_{2,(a,b)}$
and $G_{2,(b,c)}$ (respectively). There is a natural injection $\phi: G_{2,\ell} \rightarrow F_{[a,b]} \times F_{[b,c]}$
defined by the rule $h \mapsto (h_{\mid [a,b]}, h_{\mid [b,c]})$. The map $\phi$ is not surjective, since the left and right derivatives of $h$ at $b$ are necessarily the same, but 
it is always possible to find $(h_{1},h_{2}) \in F_{[a,b]} \times F_{[b,c]}$ such that $h'_{1}(b) \neq h'_{2}(b)$. (The derivatives in question are from the left and right, respectively.) 

Now define a homomorphism $\theta: F_{[a,b]} \times F_{[b,c]} \rightarrow \mathbb{Z}$ by the rule
$\theta(h_{1},h_{2}) = \log_{2}(h'_{1}(b)) -\log_{2}(h'_{2}(b))$. This is a homomorphism essentially by results of \cite{CFP}, and the kernel consists of the pairs 
$(h_{1},h_{2})$ where the left and right derivatives at $b$ agree. This is precisely $\phi(G_{2,\ell})$.    

Now we appeal to Theorem 3.2 from \cite{BGK}, which computes all of the Sigma invariants $\Sigma^{m}(F^{n})$, for all $m$ and $n$. These invariants completely encode the finiteness properties $F_{m}$ of $Ker(\theta)$, for all homomorphisms $\theta: F^{n} \rightarrow \mathbb{Z}^{d}$. In particular, the finiteness properties of $Ker (\theta) = G_{2,\ell}$ can be entirely determined from \cite{BGK} using the isomorphism
$F_{[a,b]} \times F_{[b,c]} \cong F^{2}$, and the $F_{\infty}$ property for $G_{2,\ell}$ can thus be proved. This argument can be extended to arbitrary $\ell$ without any essential change.

\bibliographystyle{plain}
\bibliography{biblio}

\end{document}